\documentclass[12pt,a4paper]{amsart}
\usepackage[foot]{amsaddr}

\makeatletter
\usepackage[english]{babel}
\usepackage[utf8]{inputenc}
\usepackage{csquotes}
\usepackage[T1]{fontenc}
\usepackage{latexsym}
\usepackage[leqno]{amsmath}
\usepackage{amssymb,amsthm,amsfonts}
\usepackage{mathtools}
\usepackage{mathrsfs}
\usepackage{dsfont}
\usepackage{stmaryrd}
\usepackage{cancel}
\usepackage{xcolor}
\usepackage{comment}

\usepackage{todonotes}

\usepackage{geometry}
\usepackage{nicefrac}

\usepackage{enumerate}
\usepackage[shortlabels]{enumitem}
\setlist[enumerate]{label=(\alph*),font=\normalshape}
\setlist[itemize]{font=\normalshape}
\let\originalitem\item
\renewcommand{\item}[1][]{%
	\if\relax\detokenize{#1}\relax%
		\originalitem%
	\else%
		\originalitem[#1]%
		\phantomsection
		\def\@currentlabel{#1}
	\fi%
}

\usepackage{hyperref}

\usepackage[noabbrev,capitalize]{cleveref}

\newcommand{\refcheckize}[1]{%
  \expandafter\let\csname @@\string#1\endcsname#1%
  \expandafter\DeclareRobustCommand\csname relax\string#1\endcsname[1]{%
    \csname @@\string#1\endcsname{##1}\wrtusdrf{##1}}%
  \expandafter\let\expandafter#1\csname relax\string#1\endcsname
}
\refcheckize{\cref}
\refcheckize{\Cref}

\usepackage[style=numeric-comp]{biblatex}

\let\originalleft\left
\let\originalright\right
\renewcommand{\left}{\mathopen{}\mathclose\bgroup\originalleft}
\renewcommand{\right}{\aftergroup\egroup\originalright}

\definecolor{kitgreen}{RGB}{0,150,130}
\definecolor{kitblue}{RGB}{70,100,170}
\definecolor{kitmaygreen}{RGB}{140,182,60}
\definecolor{kityellow}{RGB}{252,229,0}
\definecolor{kitorange}{RGB}{223,155,27}
\definecolor{kitbrown}{RGB}{167,130,46}
\definecolor{kitred}{RGB}{162,34,35}
\definecolor{kitpurple}{RGB}{163,16,124}
\definecolor{kitcyanblue}{RGB}{35,161,224}

\allowdisplaybreaks

\theoremstyle{plain}
\newtheorem{theorem}{Theorem}[section]
\newtheorem{definition}[theorem]{Definition}
 
\newtheorem{lemma}[theorem]{Lemma}

\theoremstyle{definition}
\newtheorem{remark}[theorem]{Remark}
\newtheorem{example}[theorem]{Example}
\newcommand{\R}{\mathbb{R}} 
 
\newcommand{\Z}{\mathbb{Z}}

\newcommand{\N}{\mathbb{N}} 
\newcommand{\No}{\mathbb{N}_0}
\newcommand{\Nodd}{\mathbb{N}_\mathrm{odd}}

\newcommand{\id}{I} % identity
\DeclareMathAlphabet{\othermathbb}{U}{bbold}{m}{n}

\newcommand{\der}[2][]{\@ifnextchar\der{\,#1\mathrm{d}#2\!}{\,#1\mathrm{d}#2}}

\DeclareMathOperator*{\esssup}{ess~sup}

\newcommand{\ee}{\mathrm{e}}
\newcommand{\ii}{\mathrm{i}}
\let\eps\varepsilon

\newcommand{\set}[1]{{\left\{ #1 \right\}}}

\newcommand{\abs}[1]{\left\lvert #1 \right\rvert}

\newcommand{\norm}[1]{\left\lVert #1 \right\rVert}

\newcommand{\nnorm}[1]{\left\vert\kern-0.25ex\left\vert\kern-0.25ex\left\vert #1 \right\vert\kern-0.25ex\right\vert\kern-0.25ex\right\vert}

\newcommand{\ip}[2]{\left\langle #1 , #2 \right\rangle}

\newcommand{\pdv}[3][]{%
	\if\relax\detokenize{#1}\relax%
		\frac{\partial #2}{\partial #3}%
	\else%
		\frac{\partial^{#1} #2}{\partial {#3}^{#1}}%
	\fi%
}
\newcommand{\dv}[3][]{%
	\if\relax\detokenize{#1}\relax%
		\frac{\mathrm{d} #2}{\mathrm{d} #3}%
	\else%
		\frac{\mathrm{d}^{#1} #2}{\mathrm{d} {#3}^{#1}}%
	\fi%
}

\newlength{\negph@wd}
\newcommand{\negphantom}[1]{%
	\ifmmode
		\mathpalette\negph@math{#1}%
	\else
		\negph@do{#1}%
	\fi
}
\newcommand{\negph@math}[2]{\negph@do{$\m@th#1#2$}}
\newcommand{\negph@do}[1]{%
	\settowidth{\negph@wd}{#1}%
	\hspace*{-\negph@wd}%
}

\newcommand{\clonelabel}[2]{\@bsphack
	\expandafter\ifx\csname r@#2\endcsname\relax
	\else\protected@write\@auxout{}{\string\newlabel{#1}%
		{\csname r@#2\endcsname}}%
	\fi
	\expandafter\ifx\csname r@#2@cref\endcsname\relax
	\else\protected@write\@auxout{}{\string\newlabel{#1@cref}%
		{\csname r@#2@cref\endcsname}}%
	\fi
	\@esphack}

\newcommand{\begin@color}[1]{\begingroup\color{#1}} 
\newcommand{\beginold}{\begin@color{red}}
\newcommand{\beginnew}{\begin@color{blue}}

\let\end@color\endgroup
\let\endold\end@color
\let\endnew\end@color

\newcommand{\dxsquare}{\dv[2]{}{x}}
\newcommand{\dtsquare}{\dv[2]{}{t}}

\newcommand{\loc}{\mathrm{loc}}
\newcommand{\per}{\mathrm{per}}
\newcommand{\ess}{\mathrm{ess}}

\newcommand{\calM}{\mathcal{M}}

\renewcommand{\S}{\mathbb{S}}

\newcommand{\Lqtvd}{L^{\tilde{q}}_{Vd}}
\newcommand{\Hs}{\mathcal{H}}
\newcommand{\dxt}{\ \mathrm{d}(x,t)}

\newcommand{\dx}{\ \mathrm{d}x}
\newcommand{\dt}{\ \mathrm{d}t}

\makeatother

\bibliography{bibliography}

\begin{document}
	\title[Rogue waves for semilinear wave equations]{Rogue waves for semilinear wave equations}

	\author{Julia Henninger$^1$}
	\email{julia.henninger@kit.edu}
	\address{$^1$Institute for Analysis, Karlsruhe Institute of Technology (KIT), D-76128 Karlsruhe, Germany}
	
	\date{\today}
	\subjclass[2000]{Primary: 35L71, 37K58; Secondary: 35A15, 35L76}
    % 35L71 2nd order semilinear hyperbolic
    % 35L76 higher order semiinear hyperbolic
    % 37K58 Variational principles and methods for infinite-dimensional Hamiltonian and Lagrangian systems
    % 35A15 Variational methods applied to PDEs
	\keywords{Semilinear hyperbolic equation, rogue waves, variational method, concentration compactness, spectral gaps}

	\begin{abstract}
		We study the semilinear wave equation $ V(x) \partial_t^2 u + d(t) M(x,\nabla_x) u=\tilde{V}(x) \tilde{d}(t) |u|^{p-1}u$ on $ \R^N \times \R$ and show the existence of solutions which are localized in space and in time, called rogue waves, by means of variational methods. 
We introduce an energy functional on a suitable Hilbert space, and provide sufficient conditions on the coefficients $V, \tilde{V}, d, \tilde{d}$, the elliptic operator $M$ and $p>1$ for the existence of a critical point. Our approach is based on a detailed analysis of the wave type operator and in particular its spectral properties. Further regularity considerations show that critical points are weak solutions to our equation. Moreover, we provide examples of the coefficients and the elliptic operator which satisfy our assumptions.

    \end{abstract}

	\maketitle
 %   \tableofcontents
%\newpage

	\section{Introduction and main result}

We study the semilinear wave equation 
\begin{align}\label{eq:main}
    V(x) \partial_t^2 u + d(t) M(x,\nabla_x) u=\tilde{V}(x) \tilde{d}(t) |u|^{p-1}u, \quad (x,t) \in \R^N \times \R
\end{align}
for the existence of rogue wave solutions under suitable conditions.
Rogue waves to \eqref{eq:main} are real-valued solutions which are localized in space and in time, i.e. $\lim_{|(x,t)| \to \infty}u(x,t)=0$.
In general, they are a rare nonlinear phenomenon and not yet well understood in a mathematically rigorous sense for wave-type equations such as \eqref{eq:main}.
The likely most well-known example for a rogue wave type solution is the Peregrine soliton. In \cite{pere} Peregrine studied the one-dimensional, cubic, focusing NLS
$\ii q_t + \tfrac{1}{2}q_{xx}+|q|^2 q=0$ for $(x,t)\in \R \times \R $
and provided an explicit solution $q(x,t)=\left(1-\frac{4(1+2\ii) t)}{1+4 x^2+4 t^2}\right) \ee^{\ii t}$ with $|q(x,t)| \to 1$ as $|(x,t)| \to \infty$. Note that this solution is localized with respect to the constant background $1$ and not $0$ as we introduced in the notion of rogue waves.
Inspired by this example, further results on rogue waves were achieved for various types of equations by different approaches. 
In the setting of integrable systems we would like to mention the books by Yang and Yang \cite{yangyang} and by Guo et al \cite{gtylw} and the contained references. There, they physically motivate numerous equations and construct rogue waves for them using different methods such as the Darboux transform, the bilinear method and inverse scattering theory. Further we point out the papers by Chen, Pelinovsky and White, where they give explicit examples of rogue waves to derivative NLS \cite{chen2}, focusing NLS \cite{chen3} and modified Korteweg–de Vries \cite{chen1}.
In the context of nonintegrable systems different techniques  are required to construct solutions and in particular rogue wave solutions. 
For the Benney-Roskes model and nonlocal NLS equation, Kevrekidis et al, proved the existence of rogue waves using a perturbative method, see \cite{kev1,kev2,kev3}. 
In the work of Schneider and Thorin, see \cite{st}, they demonstrated an approximation result for the Peregrine soliton in nonlinear dispersive wave equations. 
Further, Plum and Reichel constructed an explicit rogue wave solution for a semilinear curl-curl wave equation in dimension three \cite{pr}.
Lastly we would like to mention that there exist many contributions about rogue waves in a physical context, for example by E. Pelinovsky, Kharif and Slunyeav \cite{peli} on oceanic rogue waves.

For our purposes we work with the weighted $L^q$-spaces $L^q_V\coloneqq L^q_V(\R^N)\coloneqq L^q(\R^N, V(x) \dx )$
and $L^q_{Vd}\coloneqq L^q_{Vd}(\R^N\times \R)\coloneqq L^q(\R^N \times \R, V(x)d(t) \dxt )$ where $V$ and $d$ are measurable, bounded and non-negative functions.
We use the notation $\lesssim$ for inequalities up to constants, more precisely, $a \lesssim b$ means that there exists a constant $c>0$ such that $a \leq c b$ for $a,b \in \R$.

Our method is based on a detailed analysis of the spectral properties of the weighted wave operator 
\begin{align*}
    L\coloneqq \frac{1}{d(t)} \partial_t^2 + \frac{1}{V(x)} M(x, \nabla_x) .
\end{align*}
In general, we assume that $M$ is elliptic and $V,\tilde{V}, d, \tilde{d}$ are non-negative. 
More precisely we work under the following assumptions. 
\begin{enumerate}[label=(A\arabic*)]
    \item  $\frac{1}{V(x)}M(x,\nabla_x)$ is a self-adjoint operator on $L^2_V(\R^N)$ with distinct eigenvalues $0 \leq \nu_1 < \nu_2 < \nu_3 < \dots$. For each eigenvalue $\nu_n$ we denote the dimension of the eigenspace by $d_n$. The corresponding complete orthonormal system of $L^2_V(\R^N)$ is given by the associated eigenfunctions $(\varphi_{n,l})_{n \in \N, 1 \leq l \leq d_n}$.\label{as:MV} \label{as:first}
    \item There exists $K \in \N$ such that the eigenvalues $\nu_n$ satisfy $\# \set{n \colon \sqrt{\nu_n} \in B} \leq K$ for any interval $B$ of length $1$. \label{as:evM}
    \item There exist $\alpha \geq 0$ and $c >0$ such that $d_n \norm{\varphi_{n,l}}_\infty^2 \leq c \nu_n^{\alpha}$ for all $n \in \N$ and $1 \leq l \leq d_n$. \label{as:efestimateM}
    
    \item  $d, \tilde{d} \in L^\infty(\R)$ with $\ess \inf_\R d>0$, $\tilde{d}>0$ almost everywhere, and $d$ has locally bounded variation. \label{as:dgen}
    
    \item  There exist $T^\pm >0$, $R^\pm \in \R$ and $T^\pm$-periodic functions $d_\per^\pm \in L^\infty(\R)$ such that $d(t)=d_\per^+(t)$ for $t>R^+$ and $d(t)=d_\per^-(t)$ for $t<R^-$. \label{as:dper}
    
    \item There exists $r > \frac{1}{2}$ such that the point spectrum of 
    $- \frac{1}{d(t)} \dtsquare : H^2(\R) \subset L^2_d(\R) \to L^2_d(\R)$ fulfills
	\begin{align*}
		\sum_{\lambda \in \sigma_p(-\frac{1}{d(t)} \dtsquare)} \lambda^{-r} < \infty .
	\end{align*}
    \label{as:dpointspec}
    \item  $\displaystyle \inf\set{\abs{\sqrt{\lambda} - \sqrt{\nu_n}} \colon n \in \N, \lambda \in \sigma\left(-\frac{1}{d(t)} \dtsquare\right)} > 0 $. \label{as:last}
     \item $V, \tilde{V} \in L^\infty(\R^N)$ with $\ess \inf_{B_R(0)} V>0$ for all $R>0$ and $V \lesssim \tilde{V} \lesssim V$. \label{as:vgen}

    \item $d$ and $\tilde{d}$ satisfy additionally one of the following assumptions \label{as:9}
    \begin{enumerate}
        \item (compact case) $\lim_{|t|\to \infty}\tilde{d}(t)=0$. \label{as:a}
       \item (asymptotically periodic case) $d=d_\per$ is $T$-periodic on $\R$ and $\tilde{d}=\tilde{d}_\per + \tilde{d}_\loc$ where $\tilde{d}_\per$ is $T$-periodic and $\tilde{d}_\loc \geq 0$, $\lim_{|t|\to \infty}\tilde{d}_\loc(t)=0$.\label{as:b}
    \end{enumerate} \label{as:pert}
\end{enumerate}

We can now state our main result.
\begin{theorem}\label{thm:main}
    Let $p \in \left(1, 1+ \frac{1}{2\alpha}\right)$ and assume that \ref{as:first}--\ref{as:9} hold. Then \eqref{eq:main} has a non-trivial rogue wave solution $u$ in the sense of Definition \ref{def:sol} with $u \in L^q_{Vd}(\R^N \times \R)$ for all $q \in \left[2, 2+\frac{1}{\alpha}\right)$.
\end{theorem}

\begin{definition}\label{def:sol}
    Consider the space 
    \begin{align*}
        \mathcal{W}\coloneqq \{ u \in L^{p+1}_{Vd}(\R^N \times \R), \partial_t u \in L^2_V(\R^N \times \R), \sqrt{\tfrac{1}{V}M}u \in L^2_{Vd}(\R^N \times \R) \} .
    \end{align*}
    A function $u \in \mathcal{W}$ is called a \textit{weak solution} to \eqref{eq:main} if
    \begin{align*}
        \int_{\R^N \times \R} - (\partial_t u) (\partial_t h) V(x) + \left(\sqrt{\tfrac{1}{V}M}u \right) \left(\sqrt{\tfrac{1}{V}M} h \right) V(x) d(t) \dxt = \int_{\R^N \times \R} |u|^{p-1} u h \tilde{V}(x) \tilde{d}(t) \dxt 
    \end{align*}
    for all $h \in \mathcal{W}$.
\end{definition}

\begin{remark}
\begin{enumerate}
    \item If $V$ and $\tilde{V}$ are symmetric we can formulate \ref{as:first} also for operators which are self-adjoint on symmetric subspaces of $L^2_V(\R^N)$ that are compatible with \eqref{eq:main}, cf.\@ Remark \ref{rem:exsym}.
     \item Assumption \ref{as:evM} implies that $(\nu_n)_{n \in \N}$ grows at least quadratically.
    \item  Assumption \ref{as:last} is central for our method. It guarantees that $L$ has a spectral gap around $0$ and is therefore invertible on suitable function spaces and moreover it provides embedding properties of theses spaces. 
\end{enumerate}
\end{remark}

\begin{remark} \label{rem:specest}
    By \ref{as:last} there exists a constant $c>0$ such that
    \begin{align}\label{eq:specest}
        \left|-\lambda + \nu_n\right| = | \sqrt{\lambda} - \sqrt{\nu_n}| | \sqrt{\lambda} + \sqrt{\nu_n}| \geq  c \sqrt{\nu_n}
    \end{align}
    for all $n \in \N$ and $\lambda \in \sigma\left(-\frac{1}{d(t)} \dtsquare\right)$. This estimate is used several times throughout this paper.
\end{remark}
We continue with presenting examples for $M$, $V$ and $d$ that satisfy our conditions \ref{as:first}--\ref{as:last}. The corresponding proofs are provided in Appendix \ref{sec:apex}. 
According to assumption \ref{as:vgen} and \ref{as:9} one can construct various examples for $\tilde{V}$ and $\tilde{d}$, based on $V$ and $d$, respectively, and therefore we do not give explicit examples for them.  
\subsection{Examples for $M$ and $V$}
\begin{example}[Quantum harmonic oscillator]\label{ex:qho}
    Let $N=1$, $M=\left(- \dxsquare +x^2\right)^2$ and ${V=1}$. Then assumptions \ref{as:MV}--\ref{as:efestimateM} hold with $\alpha=0$. In particular, we have $\nu_n=(2n+1)^2$, ${n \in \N \cup \{0\}}$.
\end{example}

\begin{example}[Weighted Laplacian]\label{ex:wl}
    Let $N\geq 2$, even and $M =-\Delta+ V(x)$ with ${V(x)= \frac{1}{(|x|^2+1)^2}}$. Then assumptions \ref{as:MV}--\ref{as:efestimateM} hold with $\alpha =\frac{N^2+2N-4}{2}$.  In particular, we have ${\nu_n=(2n +N-1)^2}$, $n \in \N \cup \{0\}$.
\end{example}
\begin{remark}\label{rem:exsym}
In the case of Example \ref{ex:wl}, we can study \eqref{eq:main} also in the spatial radially symmetric setting, assuming in particular $\tilde{V}$ being radially symmetric. Then assumptions \ref{as:MV}--\ref{as:efestimateM} hold with $\alpha=\frac{N-1}{2}$ and $d_n=1$ for all $n \in \N$, which provides a larger range for $p$. Nevertheless it is open if our solutions, obtained in the setting of Example \ref{ex:wl}, are radially symmetric. 
\end{remark}
With regard to assumption \ref{as:evM}, note that in Example \ref{ex:qho} and \ref{ex:wl}, we denote the eigenvalues by $\nu_n$ with $n \in \N \cup \{0\}$ as commonly found in the literature. 
\subsection{Examples for $d$}
We continue with examples for $d$. In view of assumption \ref{as:last}, we assume that $M$ and $V$ satisfy one of the Examples \ref{ex:qho} or \ref{ex:wl}, in particular, the eigenvalues are of the form $\nu_n=(2n+1)^2$.
Motivated by \cite{hor}, we consider $d$ to be a periodic step potential or a perturbation of it.
By $d_\per$ we define the positive periodic step potential
\begin{align}\label{eq:multistepV}
\begin{cases}
     d_\per(t) = a_i & \mbox{ for } t\in [\theta_{i-1}T, \theta_i T), i=1,\dots, m, \\
     d_\per(t)=d_\per(t+T) & \text{ for all } t \in \R ,
\end{cases}
\end{align}
where $0=\theta_0<\theta_1<\dots<\theta_m=1$ and $a_1,\dots, a_m>0$.

\begin{example}[periodic $m$-step potential]\label{lem:multistepcond} 

We abbreviate $q_i \coloneqq \sqrt{a_i}(\theta_{i}-\theta_{i-1})$ for $i=1,\dots, m$.
    Assume that $ q_i \in \frac{\pi}{2}\N$ for all $i=1,\dots,m$ and 
    $$
     q_{i_j}\in \frac{\pi}{2}\Nodd
    $$
    is satisfied for an even number of indices $1\leq i_1<i_2<\dots<i_{2m}\leq m$ and no others. 
    Further let
    \begin{align*}
    \alpha \coloneqq \frac{a_{i_1} a_{i_3}\cdots a_{i_{2m-1}}}{a_{i_2} a_{i_4} \cdots  a_{i_{2m}}} \not =1 .
   \end{align*}
    Then \ref{as:dgen}--\ref{as:last} hold for $ d_\per$.
\end{example}
\begin{remark}
    In the case $m=2$ the assumptions reduce to 
    \begin{align*}
        \sqrt{a_1}\theta_1, \sqrt{a_2} (1-\theta_1) \in \frac{\pi}{2} \Nodd \text{ with } a_1 \neq a_2.
    \end{align*}
\end{remark}

Beyond the purely periodic case we can also treat perturbations of it. One possible kind of perturbation is the interface of dislocated periodic potentials.
\begin{example}[Interface of dislocated periodic potentials]\label{ex:dislint}
Let $d_\per$ be given by \eqref{eq:multistepV} satisfying the conditions from Example \ref{lem:multistepcond} and for $d_0, b >0$ consider 
\begin{align*}
    d(t)=\begin{cases}
        d_\per(t),  & t <0,\\
        d_0,  & 0 \leq t < b, \\
        d_\per(t-b),  & b \leq t . \\
    \end{cases}
\end{align*}

Assume that $\sqrt{d_0}b \in \pi \N$. Then \ref{as:dgen}--\ref{as:last} hold. 
\end{example}

Another possible perturbation is the interface of two different purely periodic step potentials.
\begin{example}[Interface of different periodic potentials]\label{lem:interface} 
Let $d_\per^+ \neq d_\per^-$ both be given as in \eqref{eq:multistepV} satisfying the conditions from Example \ref{lem:multistepcond}.
Consider
\begin{align*}
    d(t) = \begin{cases}
        d^-_\per(t) , & t<0 , \\
        d^+_\per(t) , & t\geq 0 
    \end{cases}
\end{align*}
and assume additionally that  $\alpha^+, \alpha^- > 1$ or $\alpha^+, \alpha^- < 1$.
Then \ref{as:dgen}--\ref{as:last} hold.
\end{example}

Now, having concrete examples for $M$ and $V$ in hand, we can analyze the regularity properties of our solutions obtained by Theorem \ref{thm:main}. We have the following result.  
\begin{theorem}\label{them:reg}
    Let $u$ be a solution as in Theorem \ref{thm:main}.
    \begin{enumerate}
         \item Let $M$ and $V$ be as is Example \ref{ex:qho}. Then $\partial_t u, x u, x^2 u , \partial_x u , x \partial_x u , \partial_x^2 u \in L^2(\R\times \R)$.
    \item Let $M$ and $V$ be as is Example \ref{ex:wl}.
    Then $\partial_t u \in L^2_V (\R^N \times \R)$ and $\nabla_x u \in L^2_{d}(\R^N\times \R)$.
     \end{enumerate}
\end{theorem}

The paper is structured as follows: 
We use a variational method to show the existence of rogue waves solutions to \eqref{eq:main}.
In Section~\ref{sec:domainfunctional}, we develop the variational setting, i.e., we introduce an energy functional for \eqref{eq:main} on a suitable Hilbert space. 
Further we present embedding properties of the Hilbert space into weighted $L^q$-spaces, which are central for the variational argument and are proven in Appendix \ref{sec:approof}.
In Section~\ref{sec:proofmain}, we proceed with the variational method and show the existence of a nontrivial critical point to our energy functional using a saddle point technique, see Theorem \ref{thm:groundstate}. In particular, Theorem \ref{thm:main} then follows from the existence of a critical point and its regularity properties, see Lemma \ref{lem:reglin} and \ref{lem:nonlin}.
Moreover, in the case of our examples, we further analyze these regularity properties and prove Theorem \ref{them:reg}. Lastly, the proofs for the Examples of $M$, $V$ and $d$ are examined in Appendix \ref{sec:apex}. Throughout the paper we assume that \ref{as:first}--\ref{as:9} hold.

    \section{Variational setting and embedding results}\label{sec:domainfunctional}

Our goal is to find rogue wave solutions to \eqref{eq:main} as critical points of the corresponding energy functional.
According to \ref{as:first} we use the eigenfunction basis $(\varphi_{n,l})_{n \in \N, 1 \leq l \leq d_n}$ and expand $u$ by
\begin{align}\label{eq:solan}
    u(x,t)=\sum_{n \in \N} \sum_{l=1}^{d_n} u_{n,l}(t) \varphi_{n,l}(x)
\end{align}
with $u_{n,l}(t)=\ip{u(\cdot,t)}{\varphi_{n,l}(\cdot)}_{L^2_V(\R^N)} $. 
Applying formally the weighted wave operator $L$ to the ansatz \eqref{eq:solan} leads to a family of Sturm-Liouville operators 
\begin{align*}
    L_n=\frac{1}{d(t)}\dtsquare +\nu_n , \quad  n \in \N .
\end{align*}
For each $n \in \N$ we have a bilinear form $b_{L_n}$ corresponding to $L_n$ given by 
\begin{align*}
      b_{L_n}(v,v)\coloneqq \int_\R - |v'|^2 + \nu_n d(t) |v|^2 \dt, \quad v \in H^1(\R)
\end{align*}
with $\langle L_n v,v \rangle_{L^2_d(\R)}=b_{L_n}(v,v)$ for $v \in C_c^\infty(\R)$. 
Next, we construct a suitable domain to $b_{L_n}$. For that we use a spectral resolution of
\begin{align}\label{eq:specres}
    -\frac{1}{d(t)} \dtsquare : H^2(\R) \subset L^2_d(\R) \to L^2_d(\R), \quad -\frac{1}{d(t)} \dtsquare = \int_\R  \lambda \mathrm{d}P_\lambda 
\end{align}
and define  
\begin{align*}
    \sqrt{|L_n|}\coloneqq \int_\R \sqrt{|-\lambda+\nu_n|}\mathrm{d}P_\lambda 
\end{align*}
on $H^1(\R)$. 
As a consequence of \ref{as:last} we can define the norm
\begin{align*}
    \norm{v}_{\Hs_n}^2\coloneqq \norm{\sqrt{|L_n|}v}^2_{L_d^2(\R)}=b_{|L_n|}(v,v)
\end{align*}
on $H^1(\R)$. 
By definition of $b_{|L_n|}$ and Lemma \ref{lem:est} in Appendix \ref{sec:approof} the norm $\| \cdot \|_{\Hs_n}$ is equivalent to the standard norm $\| \cdot \|_{H^1(\R)}$ on $H^1(\R)$ for every fixed $n \in \N$.
Hence we denote the Hilbert space $\mathcal{H}_n\coloneqq(H^1(\R), \| \cdot \|_{\mathcal{H}_n})$.

We define the two orthogonal projections $P^\pm_n: \Hs_n \to \Hs_n$ by
\begin{align*}
    v^+ \coloneqq P^+_n[v] \coloneqq \int^{\nu_n}_{-\infty} 1 \mathrm{d}P_\lambda[v] , \quad  v^- \coloneqq P^-_n[v] \coloneqq \int_{\nu_n}^{\infty} 1 \mathrm{d}P_\lambda[v]
\end{align*}
and decompose $\Hs_n$ orthogonally such that $\Hs_n=\Hs_n^+ \oplus \Hs_n^-$, $v=v^++v^-$ with
\begin{align*}
    \norm{v}_{\Hs_n}^2=\norm{v^+}_{\Hs_n}^2+\norm{v^-}_{\Hs_n}^2 .
\end{align*}

For $v \in \Hs_n$ our bilinear form satisfies
\begin{align}\label{eq:bln}
    b_{L_n}(v,v)=\int_\R - |v'|^2 + \nu_n d(t) |v|^2 \dt = \norm{v^+}_{\Hs_n}^2-\norm{v^-}_{\Hs_n}^2 .
\end{align}
Lastly we introduce the form domain for $L$ via 
\begin{align*}
    \Hs \coloneqq \{ u \in L^2_{Vd}(\R^N \times \R): u(x,t)=\sum_{n \in \N} \sum_{l=1}^{d_n}u_{n,l}(t) \varphi_{n,l}(x), \norm{u}_{\Hs}< \infty \}
\end{align*}
with the norm given by
\begin{align*}
    \norm{u}_{\Hs}^2 \coloneqq \sum_{n \in \N} \sum_{l=1}^{d_n}\norm{u_{n,l}}_{\Hs_n}^2 = \sum_{n \in \N} \sum_{l=1}^{d_n} b_{|L_n|}(u_{n,l},u_{n,l}) .
\end{align*}
By Remark \ref{rem:specest} the space $\Hs$ embeds continuously into $L^2_{Vd}(\R^N \times \R)$. Further we have ${\Hs=\oplus_{n \in \N}\Hs_n^{d_n}}$ and the orthogonal decomposition of $\Hs_n=\Hs_n^+ \oplus \Hs_n^-$ extends to $\Hs=\Hs^+ \oplus \Hs^-$ where $\Hs^\pm=\oplus_{n \in \N}\Hs_n^{d_n,\pm}$.
The corresponding energy functional to \eqref{eq:main} is then given by 
\begin{align*}
    J(u)& \coloneqq \frac{1}{2}\sum_{n \in \N} \sum_{l=1}^{d_n} b_{L_n}(u_{n,l},u_{n,l}) - \frac{1}{p+1} \norm{u}_{L^{p+1}_{\tilde{V} \tilde{d}}(\R^N \times \R)}^{p+1} \\
    &= \frac{1}{2} \norm{u^+}_{\Hs}^2- \frac{1}{2}\norm{u^-}_{\Hs}^2 - \frac{1}{p+1} \norm{u}_{L^{p+1}_{\tilde{V} \tilde{d}}(\R^N \times \R)}^{p+1} 
\end{align*}
for $u \in \Hs$.
The following embedding result is central for our method and ensures, in particular, that $J$ is well-defined on $\Hs$.

\begin{theorem}\label{theom:emb}
    Let $\tau > 0$.
    Then $\Hs \hookrightarrow L^q_{Vd}(\R^N \times \R)$ is bounded and  $\Hs \hookrightarrow L^q_{Vd}(\R^N \times [-\tau,\tau])$ is compact for all $q \in \left[2, 2+\frac{1}{\alpha}\right)$.
\end{theorem}

By Theorem \ref{theom:emb} the embedding is only locally compact with respect to time. In order to overcome this lack of global compactness we will later use the following variant of P. L. Lion's concentration compactness result, see \cite{Willem}, in the case when $d$ is periodic.

\begin{theorem}\label{lem:cct} Let $T>0$ and $q \in \left[2,  2+\frac{1}{\alpha}\right)$ be given and let $(u_j)_{j \in \N} \subset \Hs$ with $\|u_j\|_\Hs \leq c $ for all $j \in \N$. Further assume
\begin{align*}
\lim_{j \to \infty}\sup_{m \in \Z} \int_{\R^N \times [mT,(m+1)T]}  |u_j(x,t)|^q V(x) \dxt =0 .
\end{align*}
Then, $u_j \to 0$ in $\Lqtvd(\R^N \times \R)$ as $j \to \infty$ for all $\tilde{q} \in \left(2,  2+\frac{1}{\alpha}\right)$.
\end{theorem}

The proofs of Theorem \ref{theom:emb} and \ref{lem:cct} are presented in Appendix \ref{sec:approof}.

   \section{Proof of main result}\label{sec:proofmain}
\subsection{Existence of a critical point}
In order to find solutions to \eqref{eq:main} we first show the existence of critical points of $J$.
Since $J$ is unbounded from above and below we find them as saddle points.
A suitable approach for our setting is the method of the generalized Nehari manifold, see Section 4 in \cite{szulkin_weth} by Szulkin and Weth.

The idea is to minimize $J$ on the set 
\begin{align*}
    \mathcal{M}\coloneqq \set{u \in \Hs \setminus \Hs^- : J'(u)[w]=0 \text{ for all } w \in \R u+\Hs^-}
\end{align*}
which is called the generalized Nehari manifold.
Note that $\calM$ is a natural constraint for $J$ and does not generate a Lagrange multiplier. 
In order to apply the results from \cite{szulkin_weth} we first observe that our functional $J$ can be written as $J=J_0-J_1$ where 
\begin{align}\label{eq:J=J0-J1}
    J_0(u)= \frac{1}{2} \norm{u^+}_{\Hs^+}^2 - \frac{1}{2} \norm{u^-}_{\Hs^-}^2  , \quad J_1(u)=  \frac{1}{p+1} \int_{\R^N \times \R } |u(x,t)|^{p+1} \tilde{V}(x) \tilde{d}(t) \dxt .
\end{align}
Further, $J$ is a $C^1$-functional on $\Hs$ with 
\begin{align*}
    J'(u)[v]=\langle u^+, v^+ \rangle_{\Hs^+} - \langle u^-, v^- \rangle_{\Hs^-} - \int_{\R^N \times \R} |u(x,t)|^{p-1} u(x,t) v(x,t) \tilde{V}(x) \tilde{d}(t) \dxt
\end{align*}
for all $u,v \in \Hs$.
The following lemma guarantees that the conditions $(B_1)$, $(B_2)$, $(i)$, $(ii)$ and in the case of \ref{as:9} \ref{as:a} also assumption $(iii)$ from Theorem 35 in \cite{szulkin_weth} are satisfied.
\begin{lemma}\label{lem:vorfunkt}
\begin{enumerate}
\item$J_1$ is weakly lower semicontinuous, 
\begin{align*} 
J_1(0)=0 \quad \text{ and } \quad \frac{1}{2}J_1'(u)[u]> J_1(u)>0 \text{ for } u\neq 0.
\end{align*} \label{lem:it:wlsc}
\item $\lim_{u\to 0} \frac{J_1'(u)}{\|u\|_{\mathcal{H}}}=0$ and $\lim_{u\to 0} \frac{J_1(u)}{\|u\|^2_{\mathcal{H}}}=0$. \label{lem:it:limzero}
\item For a weakly compact set $U\subset \mathcal{H}\setminus\{0\}$ we have $\lim_{s\to\infty} \frac{J_1(s u)}{s^2} = \infty$ uniformly w.r.t. $u\in U$. \label{lem:it:wcss}
\item In case \ref{as:a} of assumption \ref{as:9} the map $u \mapsto J_1'(u)$ is completely continuous from $\mathcal{H}$ to $\mathcal{H}'$. \label{lem:it:complcont}
\item For each $w\in\mathcal{H}\setminus \mathcal{H}^-$ let $\mathcal{H}(w)=\R_{\geq 0}w+ \mathcal{H}^-$. Then there exists a unique nontrivial critical point $m(w)$ of $J|_{\mathcal{H}(w)}$. Moreover, $m(w)\in \mathcal{M}$ is the unique global maximizer of $J|_{\mathcal{H}(w)}$ as well as $J(m(w))>0$. \label{lem:it:uniquecritpoint}
\item There exists $\delta>0$ such that $\| m(w)^+\|_{\mathcal{H}} \geq \delta$ for all $w\in \mathcal{H}\setminus \mathcal{H}^-$. \label{lem:it:mbbdfromblw}
    \end{enumerate}
\end{lemma}

\begin{proof}
    The proofs are essentially the same as in \cite{maier_reichel_schneider} and \cite{hor}. More precisely, by using \ref{as:dgen} and \ref{as:vgen} we can directly follow the lines of the proof of Lemma 5.1 in \cite{maier_reichel_schneider} to show  \ref{lem:it:wlsc}--\ref{lem:it:wcss}.
    For \ref{lem:it:complcont} we use that $\tilde{d}$ is localized and that the embedding $\Hs \hookrightarrow L^{p+1}_{Vd}(\R^N \times [-\tau, \tau])$ is compact by Theorem~\ref{theom:emb} for $\tau >0$, see Lemma 2.3 in \cite{hor}.
    For \ref{lem:it:uniquecritpoint} we can follow the proof of Proposition 39 in \cite{szulkin_weth}. Lastly the proof of \ref{lem:it:mbbdfromblw} is the same as the proof of Lemma 5.2 $b)$ in \cite{maier_reichel_schneider}.
\end{proof}
To obtain a candidate for a critical point, the following lemma is very useful.
\begin{lemma}\label{lem:5.3}\label{lem:boundedness_of_PS}
\begin{enumerate}
    \item   Any Palais-Smale sequence $(u_j)_{j \in \N} \subset \Hs$ for $J$ is bounded. \label{lem:it:psbdd}
    \item   There exist constants $C, \eps>0$ such that 
\begin{align}\label{eq:loc:manifold_est}
	\eps\leq \|u\|_\mathcal{H} \leq C J(u)^\frac{p}{p+1} \mbox{ for all } u\in \mathcal{M}.
\end{align}
\end{enumerate}
  \end{lemma}
\begin{proof}
 We need the following identities. First we have
\begin{align}\label{eq:ableitungj}
 J(u)-\tfrac12 J'(u)[u] = \frac{p-1}{2(p+1)} \int_{\R^N\times\R} |u|^{p+1} \tilde{V}(x) \tilde{d}(t) \der (x,t) \text{ for } u \in \Hs
\end{align}
 and in particular 
 \begin{align} \label{eq:uinM}
 J(u) = \frac{p-1}{2(p+1)} \int_{\R^N\times\R} |u|^{p+1} \tilde{V}(x) \tilde{d}(t) \der (x,t) \text{ for } u \in \calM .
\end{align}
Further for $u=u^+ + u^- \in \Hs$ we have
\begin{align}\label{eq:uplus}
J'(u)[u^+] &= \|u^+\|_\mathcal{H}^2 - \int_{\R^N\times\R}  |u|^{p-1} u u^+ \tilde{V}(x) \tilde{d}(t) \der (x,t) ,
\end{align}
\begin{align}\label{eq:uminus}
-J'(u)[u^-] &= \|u^-\|_\mathcal{H}^2 + \int_{\R^N\times\R} |u|^{p-1} u u^- \tilde{V}(x) \tilde{d}(t) \der (x,t) .
    \end{align}

Now, let $(u_j)_{j \in \N} \subset \Hs$ be a Palais-Smale sequence for $J$.
Then from \eqref{eq:uplus}, \eqref{eq:uminus}, assumptions \ref{as:dgen} and \ref{as:vgen}, Hölder's inequality and Theorem \ref{theom:emb} we infer 
$$
	\|u_j^\pm\|_\mathcal{H}^2 \lesssim \left( \|\tilde{d}\|_\infty^\frac{1}{p+1} \left(\int_{\R^N\times\R} |u_j|^{p+1}\tilde{V}(x) \tilde{d}(t) \der (x,t)\right)^\frac{p}{p+1} + o(1)\right)\|u_j^\pm\|_\mathcal{H}
$$
and hence
\begin{equation} \label{eq:normbound}
	\|u_j\|_\mathcal{H} \lesssim  \left(\int_{\R^N\times\R} |u_j|^{p+1}\tilde{V}(x) \tilde{d}(t) \der (x,t)\right)^\frac{p}{p+1} + 1.
\end{equation}
Together with \eqref{eq:ableitungj} we obtain
\begin{align*}
\int_{\R^N\times\R} |u_j|^{p+1}\tilde{V}(x) \tilde{d}(t)\der (x,t) &\lesssim J(u_j)+ o(1) \|u_j\|_\mathcal{H}  \\
& \lesssim J(u_j) + \left(\int_{\R^N\times\R} |u_j|^{p+1}\tilde{V}(x) \tilde{d}(t)\der (x,t)\right)^\frac{p}{p+1} + 1.
\end{align*}
Therefore we have
$$
\int_{\R^N\times\R} |u_j|^{p+1}\tilde{V}(x) \tilde{d}(t)\der (x,t) \lesssim J(u_j)+1
$$
and again using \eqref{eq:normbound} it follows that
$$
\|u_j\|_\mathcal{H}^\frac{p+1}{p} \lesssim J(u_j)+1 .
$$
This completes the proof of \ref{lem:it:psbdd}.

Now let $u \in \calM$. Then $J'(u)[u^+]=J'(u)[u-u^-]=0$ and $J'(u)[u^-]=0$. Similarly as in the proof of \ref{lem:it:psbdd} we have by \eqref{eq:uplus}, \eqref{eq:uminus}, assumptions \ref{as:dgen} and \ref{as:vgen}, Hölder's inequality and Theorem \ref{theom:emb}
	\begin{align*}
	\|u^\pm\|^2_\mathcal{H}
	& \leq \|\tilde{d}\|_\infty^\frac{1}{p+1} \left(\int_{\R^N\times\R}  |u|^{p+1} \tilde{V}(x) \tilde{d}(t) \der (x,t)\right)^\frac{p}{p+1} \|u^\pm\|_{L^{p+1}_V}\\
	& \lesssim \|\tilde{d}\|_\infty^\frac{1}{p+1} \left(\int_{\R^N\times\R}  |u|^{p+1}\tilde{V}(x) \tilde{d}(t) \der (x,t)\right)^\frac{p}{p+1} \|u^\pm\|_\mathcal{H}.
	\end{align*}

Together with \eqref{eq:uinM} we obtain
\begin{align*}
    \norm{u}_\Hs \lesssim J(u)^{\frac{p}{p+1}} \text{ for } u \in \calM .
\end{align*}
This shows the second inequality in \eqref{eq:loc:manifold_est}.
On the other hand, by \eqref{eq:uinM} and Theorem \ref{theom:emb} we have $J(u)\lesssim \|u\|_\mathcal{H}^{p+1}$ for $u\in\mathcal{M}$ and since $u\not=0$ also the first inequality of  \eqref{eq:loc:manifold_est} holds.
\end{proof}

\begin{remark}\label{rem:4.2}
Functions of the form
\begin{align*}
    u(x,t)=\sum_{n=1}^K \sum_{l=1}^{d_n}u_{n,l}(t) \varphi_{n,l}(x)
\end{align*}
where all $u_{n,l} \in H^1(\R)$ have compact support are dense in $\Hs$.
\end{remark}
We can now state and prove our central existence result.
\begin{theorem}\label{thm:groundstate} The functional $J$ admits a ground state.
\end{theorem}
\begin{proof}
We use Theorem 35 in \cite{szulkin_weth} to show the existence of a ground state for our functional $J$. 
First we observe that by Lemma \ref{lem:vorfunkt} the conditions $(B_1), (B_2), (i)$ and $(ii)$ in Theorem 35 are satisfied.
In the case of assumption \ref{as:9} \ref{as:a} also condition $(iii)$ holds by Lemma \ref{lem:vorfunkt} and we directly obtain the existence of a ground state of $J$ on $\calM$ by Theorem 35.

For the rest of the proof we assume that \ref{as:9} \ref{as:b} holds.
By Theorem 35 in \cite{szulkin_weth} we obtain a minimizing  Palais-Smale sequence $(u_j)_{j \in \N}$ of $J$ in $\calM$. Lemma \ref{lem:5.3} guarantees that $(u_j)_{j \in \N}$ is bounded.
Thus, there exist $u \in \Hs$  and a subsequence (again denoted by $(u_j)_{j \in \N}$) such that $u_j \rightharpoonup u$ in $\Hs$ as $j \to \infty$.
By weak convergence we have $J_0'(u_j)[v] \to J_0'(u)[v]$ as $j \to \infty$ for all $v \in \Hs$.
Further, by Theorem \ref{theom:emb} we have 
\begin{align*}
J_1'(u_j)[v]= \int_{\R^N \times \R} |u_j|^{p-1} u_j v \tilde{V}(x) \tilde{d}(t) \dxt \to J_1'(u)[v]
\end{align*} 
as $j \to \infty$ for all $v \in \Hs$ with $\text{supp}(v) \subset \R^N \times [-\tau,\tau]$ and $\tau >0$.
By Remark \ref{rem:4.2} and the local compactness of the embedding from Theorem \ref{theom:emb} this is sufficient to conclude that $u$ is a critical point of $J$.

It remains to show that the weak limit $u$ is in $\calM$, meaning $u \notin \Hs^-$ and that it is a minimizer of $J$ on $\calM$. 
We distinguish two cases according to assumption \ref{as:9} \ref{as:b}. 
Note that the proofs are essentially contained in \cite{maier_reichel_schneider} and \cite{hor} but for the reader's convenience we sketch the arguments.

\textit{Periodic case}: We first consider $d=d_\per$ and $\tilde{d}=\tilde{d}_\per$ to be purely periodic with the same period $T$, i.e., $\tilde{d}_\loc=0$. In particular our functional $J$ has a periodic structure in time. This allows us to use concentration compactness arguments in order to bypass the lack of global compactness in time. We proceed in two steps:\\
Step 1: We show the existence of a new Palais-Smale sequence $(v_j)_{j \in \N}$ in $\Hs$ such that ${v_j \rightharpoonup v \in \calM}$ and $J{(v_j) \to \inf_{\calM} J}$ as $j \to \infty$.
Since $(u_j)_{j \in \N}$ is in $\calM$ and by Lemma \ref{lem:boundedness_of_PS}, there exists $c >0$ such that $\norm{u_j}_{L^{p+1}_{Vd}}\geq c >0$ for all $j \in \N$ and by Theorem \ref{lem:cct} follows 
\begin{align}\label{eq:46}
\liminf_{j \to \infty} \sup_{m \in \Z} \int_{\R^N \times [mT,(m+1)T]} |u_j|^{2}  V(x) \dxt >0 .
\end{align}
Therefore we find $\delta >0$, a sequence $(m_j)_{j \in \N} \in \Z$ of translations and a subsequence of $(u_j)_{j \in \N}$ (again denoted by $(u_j)_{j \in \N}$) such that
\begin{align*}
\int_{\R^N \times [m_jT,(m_j+1)T]} |u_j|^{2}  V(x) \dxt \geq \delta >0
\end{align*}
for all $j \in \N$. 
By setting
\begin{align*}
v_j(x,t)\coloneqq u_j(x,t+m_jT)
\end{align*}
we obtain a new Palais-Smale sequence $(v_j)_{j\in \N}$ for $J$ with $J(v_j)=J(u_j)$, ${\lim_{j \to \infty} J(v_j)=\inf_{\calM}J}$, and
\begin{align}\label{eq:vjest}
\int_{\R^N \times [0,T]} |v_j|^{2}  V(x) \dxt= \int_{\R^N \times [m_jT,(m_j+1)T]} |u_j|^{2}  V(x) \dxt \geq \delta >0
\end{align}
for all $j \in \N$. 
Again by Lemma \ref{lem:5.3} we know that $(v_j)_{j \in \N}$ is bounded in $\Hs$, thus $v_j \rightharpoonup v \in \Hs$ (for a subsequence) and from Theorem \ref{theom:emb} and \eqref{eq:vjest} we know $v \neq 0$.
The property $J'(v)=0$ follows from the observation from the beginning of the proof.
It is left to prove $v^+ \neq 0$.
For that we assume $v^+=0$, i.e. $v=v^-$. We test $J'(v)=0$ with $v$ and conclude
\begin{align*}
0=J'(v)[v] =-\|v^-\|_\Hs^2 -\int_{\R^N \times \R} |v|^{p+1} \tilde{V}(x) \tilde{d}(t) \dxt <0.
\end{align*}
This is a contradiction and thus our assumption is false. Therefore $v \in \calM$.

 Step 2: We show that $v$ minimizes $J$ on $\calM$.
Since $v \in \calM$ we obviously have $J(v) \geq \inf_\calM J$.
On the other hand $v_j \to v$ pointwise a.e. (for a subsequence) and the reverse inequality follows by Fatou's Lemma, i.e.
\begin{align*}
\inf_\calM J &= \lim_{j \to \infty} J(v_j) - \frac{1}{2} J'(v_j)[v_j]= \frac{p-1}{2(p+1)} \lim_{j \to \infty} \int_{\R^N \times \R} |v_j|^{p+1} \tilde{V}(x) \tilde{d}(t) \dxt \\
&\geq \frac{p-1}{2(p+1)}  \int_{\R^N \times \R} |v|^{p+1}\tilde{V}(x) \tilde{d}(t) \dxt = J(v) - \frac{1}{2} J'(v)[v]=J(v) .
\end{align*}

\textit{Perturbed periodic case}: In this case we have $d=d_\per$ and $\tilde{d}=\tilde{d}_\per+\tilde{d}_\loc$ with $\tilde{d}_\per , \tilde{d}_\loc \neq 0$. We consider the functional $J=J_0-J_1$ as introduced in \eqref{eq:J=J0-J1} and the auxiliary $J^\per = J_0 - J_1^\per$ with $J_0$ as before and 
\begin{align*}
    J_1^\per(u) = \frac{1}{p+1}\int_{\R^N\times\R} |u|^{p+1} \tilde{V}(x) \tilde{d}_\per (t) \der (x,t).
\end{align*}
By the previous case, $J^\per$ attains its minimum 
at the ground state level $c^\per = \inf\{ J^\per(u): u \in \mathcal{M}^\per\}=J^\per(u^\per) $ on the corresponding Nehari manifold
\begin{align*}
    \mathcal{M}^\per = \{ u\in \mathcal{H}\setminus\mathcal{H}^-: {J^\per}'(u)[w] = 0 \mbox{ for all }
 w\in [u]+\mathcal{H}^-\} .
\end{align*}
With this information in hand we will show that $J$ attains its minimum on $\mathcal{M}$ at the ground state level $c$. 
First, since $\tilde{d}_\per \leq \tilde{d}_\per + \tilde{d}_\loc$, we have $J(u)\leq J^\per(u)$ for all $u \in \mathcal{H}$. Next, according to Lemma \ref{lem:vorfunkt} we introduce the maps $m: \mathcal{H}\setminus \mathcal{H}^- \to \mathcal{M}$ and $m^\per: \mathcal{H}\setminus \mathcal{H}^- \to \mathcal{M}^\per$ and we have $m^\per(u^\per)=u^\per \not \in \mathcal{H}^-$. We define $u \coloneqq m(u^\per) \in [0,\infty)u^\per + \mathcal{H}^-$ and from Lemma \ref{lem:vorfunkt} we conclude $J^\per(u) \leq J^\per(m^\per(u^\per))= J^\per(u^\per)=c^\per$. In total we have 
\begin{align*}
    c \leq J(u) \leq J^\per(u) \leq J^\per(u^\per)=c^\per .
\end{align*} 
In the case $c=c^\per$ we have $J(u)=J^\per(u)=J^\per(u^\per)$ and this implies in particular $\tilde{d}_\loc u =0$ and $u=u^\per$ by the uniqueness of the maximizer from Lemma \ref{lem:vorfunkt}. 
Lastly, we analyze the case $c<c^\per$ and show that $c$ is attained.
Let $(u_j)_{j \in \N}$ be the minimizing Palais-Smale sequence of $J$ on $\mathcal{M}$ from the beginning of the proof. By Lemma \ref{lem:boundedness_of_PS}, $u_j \not \to 0$ in $L^{p+1}_{Vd}(\R^N\times \R)$ as $j \to \infty$ and by Lemma \ref{lem:cct} there exists $\delta>0$, a sequence of translations $(m_j)_{j \in \N}$ and a subsequence of $(u_j)_{j \in \N}$ (again denoted by $(u_j)_{j \in \N}$) such that
\begin{align}\label{eq:estbelow}
    \int_{\R^N \times [m_jT,(m_j+1)T]}|u_j|^{p+1}V(x) \dxt \geq \delta >0
\end{align}
for all $j \in \N$.
Now we assume there exists a subsequence of $(u_j)_{j \in \N}$ (again denoted by $(u_j)_{j \in \N}$) such that $(u_j)_{j\in \N}\to 0$ in $L^{p+1}_{Vd, \loc}(\R^N \times \R)$ as $j \to \infty$. Then from \eqref{eq:estbelow} we conclude $|m_j| \to \infty$ as $j \to \infty$.
The sequence $v_j(x,t)\coloneqq u_j(x,t-m_j T)$ is again bounded in $\mathcal{H}$, and hence there exists a subsequence (again denoted by $(v_j)_{j \in \N})$ and $v \in \mathcal{H}$ such that $v_j \rightharpoonup v$ in $\mathcal{H}$, $v_j \to v$ in $L^{p+1}_{Vd,\loc}(\R^N \times \R)$ and pointwise a.e., as $j \to \infty$. In particular, $v \neq 0$ by \eqref{eq:estbelow} and Theorem~\ref{theom:emb}.
Next, we show that $v \in \mathcal{M}^\per$ is a nontrivial ciritical point of $J^\per$. For that, we take $\varphi \in \mathcal{M}$ with compact support according to Remark \ref{rem:4.2} and define $\varphi_j(x,t)\coloneqq \varphi(x,t+m_jT)$. Since $(u_j)_{j \in \N}$ is a Palais-Smale sequence of $J$, we have $J'(u_j)[\varphi_j] \to 0$ as $j \to \infty$. Further we calculate
\begin{align*}
    J'(u_j)[\varphi_j]&=J_0'(u_j)[\varphi_j]-J_1'(u_j)[\varphi_j] \\
    &=J_0'(u_j)[\varphi_j]- {J_1^\per}' (u_j)[\varphi_j]- \int_{\R^N \times \R} |u_j|^{p-1}u_j \varphi_j \tilde{V}(x) \tilde{d}_\loc (t) \dxt \\
    &=J_0'(v_j)[\varphi]- {J_1^\per}' (v_j)[\varphi]- \int_{\R^N \times \R} |v_j|^{p-1}v_j \varphi \tilde{V}(x) \tilde{d}_\loc (t-m_jT) \dxt.
\end{align*}
Since $v_j\pm \rightharpoonup v^\pm$ in $\mathcal{H}$, $v_j \to v$ in $L^{p+1}_{Vd, \loc}(\R^N \times R)$, $\varphi$ has compact support and $\tilde{d}_\loc$ is localized, we conclude $ J'(u_j)[\varphi_j] \to {J^\per}'(v)[\varphi] $ as $j \to \infty$ and hence $v \in \mathcal{M}$ is a nontrivial critical point of $J^\per$.
Lastly we have
\begin{align*}
    c^\per\leq J^\per(v)& = J^\per(v)-\frac{1}{2} {J^\per}'(v)[v] \\
    &= \frac{p-1}{2(p+1)} \int_{\R^N\times\R} \tilde{d}_\per(t)\tilde{V}(x) |v|^{p+1} \der (x,t) \\
    & \leq \frac{p-1}{2(p+1)} \liminf_{j\in\N} \int_{\R^N\times\R} \tilde{d}_\per(t) \tilde{V}(x) |v_j|^{p+1} \der (x,t) \\
    & \leq \frac{p-1}{2(p+1)} \liminf_{j\in\N} \int_{\R^N\times\R} \tilde{d}(t) \tilde{V}(x) |v_j|^{p+1} \der (x,t) \\
    & = \liminf_{j\in\N} \left(J(u_j)-\frac{1}{2} J'(u_j)[u_n]\right) \\
    & = c.
\end{align*}
This is a contradiction to the inequality $c < c^\per$ from the considerations above. Hence our assumption was wrong and $u_j \not \to 0$ in $L^{p+1}_{Vd, \loc}(\R^N \times \R)$. From this we can conclude that there exists $u \in \mathcal{M}$ and a subsequence of $(u_j)_{j \in \N}$ (again denoted by $(u_j)_{j \in \N}$ such that $u_j \rightharpoonup u$ as $j \to \infty$ and $c=\inf_\mathcal{M}J=J(u)$.
\end{proof}

\subsection{Regularity}\label{subsec:reg}
So far the solution $u$ to \eqref{eq:main} that we obtained from Theorem \ref{thm:main} lies in $\Hs$ and by Theorem \ref{theom:emb} also in $L^q_{Vd}(\R^N \times \R)$ for $q \in \left[2, 2+\frac{1}{\alpha}\right)$.
In the following we present the proofs of further regularity properties of $u$ as stated in Theorem \ref{thm:main} and Theorem \ref{them:reg}.
Some parts are based on the ideas of Chapter 6 in \cite{maier_reichel_schneider} and Chapter 2 in \cite{hor}. 
\medskip

First, we study the linear problem
\begin{align}\label{eq:lineq}
    V(x) \partial_t^2 w + d(t) M(x,\nabla_x) w = f(x,t) V(x) d(t) .
\end{align}
\begin{lemma}\label{lem:reglin}
Let $f \in \mathcal{H}'$. Then there exists a unique solution $w \in \mathcal{H}$ to
\begin{align}\label{eq:linprob}
    \sum_{n\in \N} \sum_{l=1}^{d_n}b_{L_n}(w_{n,l},\phi_{n,l}) 
	= \langle f, \phi \rangle_{\mathcal{H}' \times \mathcal{H}} \quad \text{ for all } \phi \in \mathcal{H}
\end{align}
with $\norm{w}_\mathcal{H}=\norm{f}_{\mathcal{H}'}$.\\
Further, if $f \in L^2_{Vd}(\R^N \times \R)$ then ${\sqrt{\frac{1}{V}M}w \in L^2_{Vd}(\R^N \times \R)}$ and ${\partial_t w\in L^2_V(\R^N\times\R)}$. 
\end{lemma}

\begin{proof}
As in Lemma 6.1 in \cite{maier_reichel_schneider}, from the decomposition of $\Hs= \Hs^+ \oplus \Hs^-$ and ${\Hs'= (\Hs^+)' \oplus (\Hs^-)'}$ together with the Riesz representation theorem we infer the first statement.
For the proof of the second statement we use the spectral resolution as in \eqref{eq:specres}
to expand 
\begin{align*}
	f(x, t) =\sum_{n \in \N} \sum_{l=1}^{d_n} \varphi_{n,l}(x) \left(\int_{\R} \der P_\lambda[ f_{n,l}]\right)(t)
\end{align*}
so that the solution $w$ to \eqref{eq:linprob} is of the form
\begin{align*}
	w(x,t) =\sum_{n \in \N} \sum_{l=1}^{d_n}\varphi_{n,l}(x) \left(\int_{\R} \frac{1}{-\lambda + \nu_n} \der P_\lambda[ f_{n,l}]\right) (t) .
\end{align*}
Further we have
\begin{align*}
	\sqrt{\frac{1}{V(x)}M(x, \nabla_x)} \ w(x, t) = \sum_{n \in \N}  \sum_{l=1}^{d_n}\varphi_{n,l}(x) \left( \int_{\R} \frac{\sqrt{\nu_n}}{-\lambda + \nu_n} \der P_\lambda[f_{n,l}] \right)(t) .
\end{align*}

By Remark \ref{rem:specest} we know
\begin{align*}
    \frac{\sqrt{\nu_n}}{|\lambda - \nu_n|}\leq c
\end{align*}
for all $n \in \N$ and $\lambda \in \sigma(- \frac{1}{d(t)} \dtsquare)$ and hence 
\begin{align*}
    \norm{ \sqrt{\frac{1}{V(x)}M(x, \nabla_x)} w}_{L^2_{Vd}(\R^N \times \R)}^2 &= \sum_{n \in \N} \sum_{l=1}^{d_n}\int_\R \frac{\nu_n}{(\lambda-\nu_n)^2 } \der \norm{P_\lambda f_{n,l}}_{L^2_d(\R)}^2  \\
    &  \leq c^2 \sum_{n \in \N} \sum_{l=1}^{d_n} \int_\R 1 \der \norm{P_\lambda f_{n,l}}_{L^2_d(\R)}^2 =  c^2 \norm{f}_{L^2_{Vd}(\R^N \times \R)}^2 .
\end{align*}

Next, by the property \eqref{eq:bln} of $b_{L_n}$ we can estimate
\begin{align*}
\sum_{n\in \N} \sum_{l=1}^{d_n} b_{L_n}(w_{n,l},w_{n,l})&=
    \sum_{n\in \N}\sum_{l=1}^{d_n}\int_\R -| w_{n,l}'|^2 + \nu_n d(t) |w_{n,l}|^2 \der t  = \norm{w^+}_{\mathcal{H}}^2-\norm{w^-}_{\mathcal{H}}^2 \\
    &\geq -\norm{w^+}_{\mathcal{H}}^2- \norm{w^-}_{\mathcal{H}}^2= - \norm{w}_\mathcal{H}^2 .
\end{align*}
This implies
\begin{align*}
    \sum_{n\in \N} \sum_{l=1}^{d_n}\int_\R | w_{n,l}'|^2 \der t &\leq \sum_{n\in \N} \sum_{l=1}^{d_n}\int_\R \nu_n d(t)| w_{n,l}|^2 \der t + \norm{w}_\mathcal{H}^2\\
	&=\norm{ \sqrt{\frac{1}{V(x)}M(x, \nabla_x)} w}_{L^2_{Vd}(\R^N \times \R)}^2 + \norm{w}_\mathcal{H}^2 ,
\end{align*}
so that we finally obtain
\begin{align*}
    \norm{ \partial_t w}_{L^2_V(\R^N \times \R)}^2 = \sum_{n \in \N}\sum_{l=1}^{d_n} \norm{w_{n,l}'}_{L^2(\R)}^2 < \infty. &\qedhere
\end{align*}
\end{proof}
Next, for both of our examples, we analyze what $\sqrt{\frac{1}{V}M}w \in L^2_{Vd}(\R^N\times \R)$ means for a function $w \in \mathcal{H}$.
We start with Example \ref{ex:qho} and prepare the corresponding result with a lemma. 
\begin{lemma}\label{lem:betteremb}
    Let $M$ and $V$ be as in Example \ref{ex:qho}. Then $w \in \Hs$ implies $x w , \partial_x w \in L^2(\R\times \R)$.
\end{lemma}
\begin{proof}
Since $V=1$ and $M=(-\partial_x^2 + x^2)^2$ we have $\sqrt{\frac{1}{V}M}=-\partial_x^2 + x^2$. Consider the associated bilinearform 
\begin{align*}
    b(z,z)=\int_\R |z'|^2+  x^2 |z| ^2 \dx 
\end{align*}
 for $z \in H^1(\R)$.
 Using the spectral resolution \eqref{eq:specres} and the estimate \eqref{eq:specest} there exists $c>0$ such that for $w\in \Hs$  we obtain
 \begin{align*}
     \|w\|_{\Hs}^2&= \sum_{n \in \N} \int_\R |-\lambda + \nu_n| \der \norm{P_\lambda  w_n}_{L^2_d(\R)}^2 \\
     &\geq c \sum_{n \in \N} \int_\R \sqrt{\nu_n} \der \norm{P_\lambda  w_n}_{L^2_d(\R)}^2 \\
     &= c\sum_{n \in \N} \sqrt{\nu_n} \|w_n\|_{L^2_d(\R)}^2 \\
     &= c\left\| \sqrt{\frac{1}{V}M}w \right\|_{L^2_{Vd}(\R \times \R)}^2 \\
     &=c\int_{\R \times \R} (|\partial_x w|^2 + x^2 |w|^2) d(t) \dxt 
 \end{align*}
and hence 
\begin{align*}
 \partial_x w , xw \in L^2(\R\times \R).
\end{align*}
\end{proof}

\begin{lemma}\label{lem:regqho}
Let $V$ and $M$ be as in Example \ref{ex:qho}. Then $w \in \Hs$ and  $\sqrt{\frac{1}{V}M}w \in L^2_{Vd}(\R \times \R)$ implies  $xw , \partial_xw,  x^2w, x \partial_x w, \partial_x^2 w \in L^2(\R\times \R)$.
\end{lemma}

\begin{proof}
    By Lemma \ref{lem:betteremb} we already know that $x w, \partial_x w \in L^2(\R\times \R)$. First we consider the function $ w(x,t)=\sum_{n=1}^K w_n(t) \varphi_n(x)$ where $w_n \in H^1(\R)$ have compact support. 
    Then using integration by parts we obtain
    \begin{align*}
          \int_{\R \times \R} \left|-\partial_x^2 w + x^2 w\right|^2 \dxt &= \int_{\R \times \R}  (\partial_x^2 w)^2 + x^4 w^2 - 2 x^2 w \partial_x^2 w \dxt         \\
          &= \int_{\R \times \R} (\partial_x^2 w)^2 + x^4 w^2 + 2 x^2 (\partial_x w)^2 - 2 w^2 \dxt
    \end{align*}
which is equivalent to 
\begin{align*}
    \norm{(-\partial_x^2 +x^2)w}_{L^2(\R \times \R)}^2+ 2 \norm{w}_{L^2(\R \times \R)}^2= \norm{\partial_x^2 w}_{L^2(\R \times \R)}^2 + \norm{x^2 w}_{L^2(\R \times \R)}^2+ 2 \norm{x \partial_x w}_{L^2(\R \times \R)}^2 .
\end{align*}
Finally, the claim follows by density according to Remark \ref{rem:4.2}.
\end{proof}

We continue with the second example for $V$ and $M$. 
\begin{lemma}\label{lem:regwl} Let $M$ and $V$ be as in Example \ref{ex:wl} and $w \in \Hs$. Then  $\sqrt{\frac{1}{V}M}w \in L^2_{Vd}(\R^N \times \R)$ implies $\nabla_x w \in L^2_{d}(\R^N\times \R)$.    
\end{lemma}

\begin{proof}
Let $D\left(\frac{1}{V}M\right) \subset \Hs$ be the domain of $\frac{1}{V}M$ and for $w,z \in D\left(\frac{1}{V}M\right) $ consider its associated form 
\begin{align*}
    b(w,z)&=\langle \tfrac{1}{V(x)}M(x,\nabla_x) w, z \rangle_{L^2_{Vd}(\R^N\times\R)}=\int_{\R^N\times\R}  (M w )z d(t) \dxt \\
   &= \int_{\R^N\times \R} ((-\Delta +V(x))w) z d(t) \dxt \\
    &=\int_{\R^N\times \R} (  \nabla_x w \cdot \nabla_x z +  w z V(x) ) d(t) \dxt .
\end{align*}
By $D(b) \subset \Hs$ we denote the domain of the form $b$.
Then it follows from Kato's second representation theorem, see Theorem 2.23 in Chapter 6, §2 in \cite{kato} that $D\left(\sqrt{\frac{1}{V}M}\right)=D(b)$ and 
\begin{align*}
    b(w,z)=\left\langle \sqrt{\tfrac{1}{V}M}w, \sqrt{\tfrac{1}{V}M}z  \right\rangle_{L^2_{Vd}(\R^N\times \R)}
\end{align*}
for all $w,z \in D\left(\sqrt{\tfrac{1}{V}M}\right) $.
Hence $\norm{\sqrt{\tfrac{1}{V}M}w}_{L^2_{Vd}}^2 < \infty$ implies in particular ${\nabla_x w \in L^2_d(\R^\N \times \R)}$.
\end{proof}

The following lemma transfers the regularity results from the linear equation \eqref{eq:lineq} to the semilinear wave equation \eqref{eq:main}.
\begin{lemma}\label{lem:nonlin}
    Let $u \in \Hs$ and $p \in \left(1, 1+\frac{1}{2 \alpha}\right)$. Then $|u|^{p-1}u \in L^2_{Vd}(\R^N \times \R)$. 
\end{lemma}
\begin{proof}
    Follows directly from Theorem \ref{theom:emb}. 
\end{proof}
\begin{proof}[Proof of Theorem \ref{thm:main}]
By Theorem \ref{thm:groundstate}, there exists a critical point $u \in \mathcal{H}$ of $J$ and by Theorem~\ref{theom:emb} we have $u \in L^{q}_{Vd}(\R^N \times \R)$ for all $q \in \left[2,2+\frac{1}{\alpha}\right)$.
Further, by our assumptions on $d, \tilde{d}$ and $V, \tilde{V}$ we can write 
\begin{align*}
|u|^{p-1}u \tilde{V}(x)\tilde{d}(t)=|u|^{p-1} u \frac{\tilde{V}(x)\tilde{d}(t)}{V(x)d(t)}V(x)d(t)
\end{align*}
with $\frac{\tilde{d}}{d} \frac{\tilde{V}}{V} \in L^\infty(\R^N \times \R)$. From Lemma \ref{lem:nonlin} we conclude $f \coloneqq |u|^{p-1}u \frac{\tilde{V}\tilde{d}}{Vd}\in L^2_{Vd}(\R^N \times \R)$ and since $u$ is a solution to \eqref{eq:lineq} for such $f$, we infer from Lemma \ref{lem:reglin} that $\partial_t u \in L^2_V(\R^N \times \R)$ and $\sqrt{\frac{1}{V}M}u \in L^2_{Vd}(\R^N \times \R)$.    
\end{proof}
\begin{proof}[Proof of Theorem \ref{them:reg}]
By Theorem \ref{thm:main}, our solution $u \in \mathcal{H}$ satisfies ${\partial_t u \in L^2_V(\R^N \times \R)}$ and $\sqrt{\frac{1}{V}M}u \in L^2_{Vd}(\R^N \times \R)$. By Lemma \ref{lem:betteremb} and \ref{lem:regqho} we obtain the result for Example \ref{ex:qho} and by Lemma \ref{lem:regwl} we obtain the result for Example \ref{ex:wl}.
\end{proof}

    \appendix
    \section{Details on the examples}\label{sec:apex}
In this section we present details on the examples of $M$, $V$ and $d$ introduced in the beginning of this paper and verify that they satisfy our assumptions. 

First we study the eigenvalue problem
\begin{align*}\label{eq:evpx}
    M\left(x, \nabla_x\right) \varphi = \nu V\left(x\right) \varphi , \quad x \in \R^N 
\end{align*}
and begin with Example \ref{ex:qho}, i.e.
\begin{equation}\label{eq:exev}
   \left(- \dxsquare +x^2\right)^2\varphi\left(x\right)=\nu \varphi\left(x\right), \quad x \in \R.
\end{equation}
By Chapter 10 in \cite{htf}, an $L^2$ orthonormal basis of eigenfunctions of \eqref{eq:exev} is given by $\varphi_n\left(x\right)= \frac{1}{\sqrt[4]{\pi}} \frac{1}{\sqrt{2^n n!}} \ee^{-\frac{1}{2}x^2}H_n\left(x\right)$, where $H_n$ are the Hermite polynomials, with eigenvalues $\nu_n=\left(2n+1\right)^2 $, $n \in \N_0$.
The eigenfunctions $\left(\varphi_n\right)_{n \in \No}$ form a complete orthonormal system of $L^2\left(\R\right)$. This verifies conditions \ref{as:first} and \ref{as:evM}.
It is furthermore known (Chapter 10 in \cite{htf}, Cramér's inequality) that $\sup_{x \in \R} |\varphi_n\left(x\right)| \leq  1 $ for all $n \in \No$. This shows that \ref{as:efestimateM} holds for $\alpha=0$.

The proof in the case of Example \ref{ex:wl}, in particular the $L^\infty$ bounds on the eigenfunctions, is involved. We therefore divide the proof in several parts.

We consider the eigenvalue problem  
\begin{align}\label{eq:evpN}
    - \Delta \varphi = \lambda \frac{1}{\left(1+|x|^2\right)^2} \varphi , \quad x \in \R^N
\end{align}
for $N \geq 2$.
The formulas for the eigenvalues, eigenfunctions and dimension of the eigenspaces to \eqref{eq:evpN} are derived in \cite{ab} and will be recalled below. 
Since the $L^\infty$ bounds on the eigenfunctions as required in assumption \ref{as:efestimateM} are not provided in \cite{ab}, we state and prove them below, see Theorem \ref{thm:efbound}.
First we present the formula for the eigenvalues and eigenfunctions and therefore introduce some notation as in \cite{ab}.
Let $\S^N\coloneqq\left\{ x \in \R^{N+1}: |x|=1 \right\}$, $M_0\coloneqq\left(0,\ldots,0,1\right) \in \S^N$ and $\S^N_\star\coloneqq\S^N \setminus \{M_0\}$. 
Furthermore we need the stereographic projection
\begin{align*}
    \Pi: \S^N_\star \to \R^N, \quad \xi \mapsto \left( \frac{\xi_1}{1-\xi_{N+1}},\ldots, \frac{\xi_N}{1-\xi_{N+1}} \right)
\end{align*}
with its inverse 
\begin{align*}
    \Pi^{-1}: \R^N \to \S^N_\star, \quad x \mapsto \left( \frac{2 x_1}{|x|^2+1},\ldots,\frac{2 x_N}{|x|^2+1},\frac{|x|^2-1}{|x|^2+1} \right) .
\end{align*}

By Theorem 3.2 in \cite{ab} the eigenvalues and eigenfunctions to \eqref{eq:evpN} are given by 
\begin{align*}
  \lambda_n^N=\left(2n+N-1\right)^2-1, \quad   \varphi_{n,l}^N\left(x\right)=\left(\frac{2}{|x|^2+1}\right)^{\frac{N-2}{2}}\mathcal{Y}_{n,l}^N\left(\Pi^{-1}\left(x\right)\right) 
\end{align*}
with $ n\in \No, 1 \leq l \leq d_n^N$ where $d_0^N=1$, $d_1^N=N+1$ and 
\begin{align}\label{eq:dimen}
    d_n^N=\binom{N+n}{N}- \binom{N+n-2}{N} .
\end{align}
Here, $\mathcal{Y}_{n,l}^N$ are the spherical harmonics of degree $n$ on $\S^N$ and for each $n$, there are $d_n^N$ such functions. In particular the eigenspace of $\lambda_n^N$ has dimension $d_n^N$ for each $n$.
Moreover, $\{\varphi_{n,l}^N\}_{n,l}$ form a complete orthonormal system of $L^2\left(\R^N, \frac{1}{\left(1+|x|^2\right)^2 }\mathrm{d}x\right)$.

In order to determine the value of $\alpha$ in assumption \ref{as:efestimateM} we need an estimate on $d_n^N$ and $L^\infty$ bounds on the eigenfunctions.
\begin{lemma}
    Let $N \geq 2$. Then there exists a constant $c_N >0$ such that $d_n^N \leq c_N n^{N-1}$ for all $n \in \N$.
\end{lemma}
\begin{proof}
    Let $N\geq 2$ be fixed. For $n \in \N$ we define
    \begin{align*}
        P\left(n\right)&\coloneqq \binom{N+n}{N}= \frac{\left(N+n\right)!}{N! n!}= \frac{1}{N!} \left(n+1\right) \cdot \ldots \cdot \left(n+N\right)\\
        &= \frac{1}{N!}\left( n^N+p_{N-1}n^{N-1}+ \ldots + p_1 n + p_0 \right)
    \end{align*}
    which is a polynomial in $n$ of degree $N$ with coefficients $1,p_{N-1},\ldots,p_0 \in \R$.
    By \eqref{eq:dimen} we have 
    \begin{align*}
        d_n^N&=P\left(n\right)-P\left(n-2\right)\\
        &= \frac{1}{N!} \left(n^N -\left(n-2\right)^N + p_{N-1}\left(n^{N-1}-\left(n-2\right)^{N-1}\right)+ \ldots + p_1\left(n -\left(n-2\right)\right)\right)      
    \end{align*}
    which is a polynomial in $n$ of degree $N-1$.
\end{proof}

The following result presents $L^\infty$ bounds on the eigenfunctions.
\begin{theorem}\label{thm:efbound}
Let $N \geq 2$ be even. 
Then there exists a constant $c=c(N)>0$ such that
\begin{align*}
     \sup_{x \in \R^N}|\varphi_{n,l}^N\left(x\right)| \leq c n^{\frac{N^2-2}{4}}
\end{align*}
for all $n \in \No$ and $1 \leq l \leq d_n^N$.
\end{theorem}

We present the proof of Theorem \ref{thm:efbound} after Theorem \ref{lem:boundsh} below.

\textbf{Change of variables:}
In order to prove the $L^\infty$ bounds it is more convenient to express the eigenfunctions in spherical coordinates: For $\xi \in \S^N$ and $\theta_1 \in [0,2\pi]$, $\theta_i \in [0,\pi]$ for all $i=2, \ldots,N$ we write
\begin{align*}
\xi=
\begin{pmatrix}
    \cos\left(\theta_1\right)& \sin\left(\theta_2\right)& \sin\left(\theta_3\right)& \ldots& \sin\left(\theta_N\right) \\
    \sin\left(\theta_1\right) &\sin\left(\theta_2\right) &\sin\left(\theta_3\right)& \ldots &\sin\left(\theta_N\right) \\
                    &\cos\left(\theta_2\right)& \sin\left(\theta_3\right) &\ldots& \sin\left(\theta_N\right) \\
                   &               & \cos\left(\theta_3\right)& \ldots& \sin\left(\theta_N\right) \\
                    &             &                &        &           \vdots     \\
                     &             &              & \cos\left(\theta_{N-1}\right)&\sin\left(\theta_N\right)\\
                                 &&&&                                  \cos\left(\theta_N\right)
\end{pmatrix}                                                   
\end{align*}
with $\frac{2}{|x|^2+1}=1-\cos\left(\theta_N\right)$.
Additionally, we change the enumeration of the eigenfunctions: Instead of using the index $l$ with $0 \leq l \leq d_n^N$ we use the $N$-tuple $\left(l_1,l_2,\ldots, l_N\right) \in \Z^N$ with $0 \leq |l_1| \leq l_2 \leq \ldots \leq l_N \coloneqq n$.
In the following, instead of using the notation $\varphi_{n,l}^N$, we will use
\begin{align*}
   \varphi_{l_1, \ldots, l_N}\left(\theta_1,\ldots, \theta_N\right)= \left(1-\cos\left(\theta_N\right)\right)^{\frac{N-2}{2}} \mathcal{Y}_{l_1, \ldots, l_N}\left(\theta_1,\ldots, \theta_N\right) .
\end{align*}
By \cite{higuchi}, the spherical harmonics $\mathcal{Y}_{l_1, \ldots, l_N}$ are given by
\begin{align*}
    \mathcal{Y}_{l_1, \ldots , l_N}\left(\theta_1, \ldots, \theta_N\right)= \frac{1}{\sqrt{2 \pi}} \ee^{\ii l_1 \theta_1} \prod_{i=2}^N {_{i}}{\overline{P}}_{l_i}^{l_{i-1}}\left(\theta_i\right) 
\end{align*}
with 
\begin{align*}
    {_{i}}{\overline{P}}_{L}^{K}\left(\theta\right)= \sqrt{\frac{2L+i-1}{2} \frac{\left(L+K+i-2\right)!}{\left(L-K\right)!}} \left(\sin\left(\theta\right)\right)^{\frac{2-i}{2}} P_{L+\frac{i-2}{2}}^{-\left(K+\frac{i-2}{2}\right)}\left(\cos\left(\theta\right)\right) 
\end{align*}
for $0 \leq K \leq L$ where $P_\xi^\mu$ are the \emph{Legendre functions} of order $\mu \in \R$ and degree $\xi \in \R$. As a function of $\theta$, $y\left(\theta\right)=P_\xi^{\mu}\left(\theta\right)$ solves the \emph{Legendre equation}
\begin{align}\label{eq:legeq}
    \frac{\der^2}{\der \theta^2} y \left(\theta\right)+ \cot\left(\theta\right) \frac{\der}{\der \theta} y\left(\theta\right) + \left( \xi\left(\xi+1\right)- \frac{\mu^2}{\sin^2\left(\theta\right)}\right)y\left(\theta\right)=0 .
\end{align}
For the special case of order $0$ and integer degree $n$, $P^0_n$ are called \emph{Legendre polynomials}.
With this notation in hand we can now state the central result which is used for the proof of Theorem \ref{thm:efbound}.
\begin{theorem}\label{lem:boundsh}
Let $N \geq 2$. Then there exists a constant $c=c(N)>0$ such that for all $2 \leq i \leq N$ and for all $0 \leq K \leq L$ we have 
    \begin{align}\label{eq:legbound}
        \sup_{\theta \in [0,\pi]}|{_{i}}{\overline{P}}_{L}^{K}\left(\theta\right) | \leq c \begin{cases}
            L^{\frac{i-1}{2}} , & \text{ if } i \text{ is even,}\\
            L^\frac{i}{2}, & \text{ if } i \text{ is odd.}
        \end{cases} 
    \end{align}
\end{theorem}

\begin{proof}[Proof of Theorem \ref{thm:efbound}]
Using the formula for the eigenfunctions and the estimates from Theorem \ref{lem:boundsh} we conclude Theorem \ref{thm:efbound}.
\end{proof}
We prove Theorem \ref{lem:boundsh} after presenting some properties of the Legendre functions and proving \eqref{eq:legbound} for the values $i=2,3,4$.

\textbf{Properties of Legendre functions:}
In the following lemma we collect some properties of the Legendre functions $P_\xi^\mu$ such as recurrence formulas and formulas for special values of $\xi$ and $\mu$ which can be found in \cite{nist}, Chapter 14.
\begin{lemma}\label{lem:rec}
For $\theta \in \left(0,\pi\right)$, $\xi,\mu \in \R$ we have
    \begin{align}\label{eq:rekrek}
        P_\xi^{\mu+2}\left(\cos\left(\theta\right)\right)+2\left(\mu+1\right) \cot\left(\theta\right) P_\xi^{\mu+1}\left(\cos\left(\theta\right)\right)+\left(\xi-\mu\right)\left(\xi+\mu+1\right) P_\xi^\mu\left(\cos\left(\theta\right)\right)=0,
    \end{align}
    \begin{align}\label{eq:recc}
        \sin\left(\theta\right)P_\xi^{\mu+1}\left(\cos\left(\theta\right)\right)-\left(\xi-\mu+1\right)P_{\xi+1}^\mu\left(\cos\left(\theta\right)\right)+\left(\xi+\mu+1\right)\cos\left(\theta\right)P_\xi^\mu\left(\cos\left(\theta\right)\right)=0,
    \end{align}
    \begin{align}\label{eq:reccc}
        \left(2 \xi +1\right) \sin\left(\theta\right) P_\xi^\mu\left( \cos\left( \theta \right)\right)+ P_{\xi+1}^{\mu+1}\left( \cos\left( \theta \right) \right)-P_{\xi-1}^{\mu+1}\left( \cos\left( \theta \right) \right)=0 ,
    \end{align}
    \begin{align}\label{eq:rec}
        \frac{2 \mu}{\sin\left(\theta\right)}P_\xi^\mu\left(\cos\left(\theta\right)\right)+P_{\xi+1}^{\mu+1}\left(\cos\left(\theta\right)\right)+\left(\xi-\mu+1\right)\left(\xi-\mu+2\right)P_{\xi+1}^{\mu-1}\left(\cos\left(\theta\right)\right)=0,
    \end{align}
    \begin{align}\label{eq:rekb}
        - \frac{\mathrm{d}}{\mathrm{d}\theta} \left( P_\xi^\mu \left(\cos\left(\theta\right)\right)\right)=\frac{1}{2}\left(\left(\xi+\mu\right)\left(\xi-\mu+1\right)P_\xi^{\mu-1}\left(\cos\left(\theta\right)\right)-P_\xi^{\mu+1}\left(\cos\left(\theta\right)\right)\right)
    \end{align}
    \begin{align}\label{eq:rekc}
        \sin\left(\theta\right) \frac{\der}{\der \theta}\left(P_\xi^\mu\left(\cos\left(\theta\right)\right)\right)=\sin\left(\theta\right) P_{\xi}^{\mu+1}\left(\cos\left(\theta\right)\right)+ \mu \cos\left(\theta\right) P_{\xi}^{\mu}\left(\cos\left(\theta\right)\right)
    \end{align}
    \begin{align}\label{eq:mueinhalb}
      P_\xi^{-\frac{1}{2}}\left(\cos\left(\theta\right)\right)=\left( \frac{2}{\pi \sin\left(\theta\right)}\right)^{\frac{1}{2}} \frac{1}{\xi+\frac{1}{2}}\sin\left(\left(\xi+\frac{1}{2}\right)\theta\right),
  \end{align}
  \begin{align}\label{eq:ximinusxi}
      P_{\xi}^{-\xi}\left(\cos\left(\theta\right)\right)=\frac{\left(\sin\left(\theta\right)\right)^\xi}{2^\xi \Gamma\left(\xi+1\right)},
  \end{align}
   where $\Gamma$ is the Gamma function.
  If $\mu=m$ is a nonnegative integer, we additionally have
  \begin{align}\label{eq:negativeind}
      P_\xi^{-m}\left(\cos\left(\theta\right)\right)=\left(-1\right)^m \frac{\Gamma\left(\xi-m+1\right)}{\Gamma\left(\xi+m+1\right)}P_\xi^m\left(\cos\left(\theta\right)\right).
  \end{align}
 \end{lemma}
For later purposes such as integration by parts, we also need the behavior of the Legendre functions at the boundary points $0$ and $\pi$. We have the following result, see \cite{nist}, Chapter 14.
\begin{lemma}\label{lem:behsing} 
If $\mu \neq -1,-2,-3,\dots$, then we have
    \begin{align*}
        P_{\xi}^{-\mu}\left(\cos\left(\theta\right)\right) \sim \frac{1}{\Gamma\left(1+\mu\right)}\left( \frac{1-\cos\left(\theta\right)}{2} \right)^{\frac{\mu}{2}}, \text{ as } \theta \to 0^+ .
    \end{align*}
    For the behavior at $\theta=\pi$, we can use the following relation
    \begin{align*}
        P_{\xi}^{-\mu}\left(-\cos\left(\theta\right)\right)= \cos\left(\left(\xi-\mu\right)\pi\right) P_{\xi}^{-\mu}\left(\cos\left(\theta\right)\right) - \frac{2}{\pi} \sin\left(\left(\xi-\mu\right)\pi\right) Q_{\xi}^{-\mu}\left(\cos\left(\theta\right)\right)
    \end{align*}
    where $Q_\xi^{-\mu}$ is the Legendre function of the second kind
\end{lemma}
\begin{remark}
In our situation, $\xi-\mu$ will be an integer and therefore the term with the Legendre function of the second kind will vanish.
\end{remark}
In general, the Legendre functions $P_\xi^\mu$ satisfy two different types of orthogonality relations, depending on whether the degree $\xi$ or the order $\mu$ is fixed. If the order $\mu$ is fixed, we have the orthogonality relation as stated in formula (2.11) in \cite{higuchi}. 
If the degree $\xi$ is fixed we have the following result.
\begin{lemma}\label{lem:ortho} Let $\mu, \Tilde{\mu} >0$ and $\xi-\mu, \xi - \Tilde{\mu} \in \N \cup \{0\}$. Then we have
    \begin{align*}
      \int_0^\pi P_\xi^{-\mu}\left(\cos\left(\theta\right)\right) P_\xi^{-\Tilde{\mu}} \left(\cos\left(\theta\right)\right) \frac{1}{\sin\left(\theta\right)} \der\theta = C\left(\xi, \mu\right) \delta_{\mu \Tilde{\mu}}
      \end{align*}
   and for the special case $\xi=L+\frac{1}{2}$, $\mu=K+\frac{1}{2}$, $0 \leq K\leq L$ the normalization constant is given by
      \begin{align*}
    C\left(L+\frac{1}{2},K+\frac{1}{2}\right)=   \frac{1}{K+\frac{1}{2}} \frac{\left(L-K\right)!}{\left(L+K+1\right)!}  .  
    \end{align*}
\end{lemma}
\begin{proof}
First, we study the orthogonality property.
    Let $\xi$ be fixed. By \eqref{eq:legeq} the Legendre function $P_\xi^{-\mu}$ satisfies 
    \begin{align}\label{eq:odemu}
      \frac{1}{\sin\left(\theta\right)}  \frac{\der}{\der \theta}\left(\sin\left(\theta\right)\frac{\der}{\der \theta} \left(P_\xi^{-\mu}\left(\cos\left(\theta\right)\right)\right)\right)+\left(\xi\left(\xi+1\right) -\frac{\mu^2}{\sin^2\left(\theta\right)}\right)P_\xi^{-\mu}\left(\cos\left(\theta\right)\right)=0
    \end{align}
    and $P_\xi^{-\Tilde{\mu}}$ satisfies
    \begin{align}\label{eq:odemutil}
       \frac{1}{\sin\left(\theta\right)}  \frac{\der}{\der \theta}\left(\sin\left(\theta\right)\frac{\der}{\der \theta} \left(P_\xi^{-\Tilde{\mu}}\left(\cos\left(\theta\right)\right)\right)\right)+\left(\xi\left(\xi+1\right)-\frac{\Tilde{\mu}^2}{\sin^2\left(\theta\right)}\right)P_\xi^{-\Tilde{\mu}}\left(\cos\left(\theta\right)\right)=0.
    \end{align}
The multiplication of \eqref{eq:odemu} with $P_\xi^{-\Tilde{\mu}}$ and \eqref{eq:odemutil} with $P_\xi^{-\mu}$ and subtraction leads to 
\begin{align*}
   &\frac{1}{\sin\left(\theta\right)} \frac{\der}{\der \theta}\left(\sin\left(\theta\right)\left(P_\xi^{-\Tilde{\mu}}\left(\cos\left(\theta\right)\right) \frac{\der}{\der \theta}\left( P_\xi^{-\mu}\left(\cos\left(\theta\right)\right)\right) - P_\xi^{-\mu}\left(\cos\left(\theta\right)\right) \frac{\der}{\der \theta} \left(P_\xi^{-\Tilde{\mu}}\left(\cos\left(\theta\right)\right)\right)\right)\right) \\
   &=\frac{\Tilde{\mu}^2-\mu^2}{\sin^2\left(\theta\right)} P_\xi^{-\mu}\left(\cos\left(\theta\right)\right) P_\xi^{-\Tilde{\mu}}\left(\cos\left(\theta\right)\right).
\end{align*}
Then we multiply this equation with $\sin\left(\theta\right)$, integrate and obtain
\begin{align*}
    &\left[ \sin\left(\theta\right) \left(P_\xi^{-\Tilde{\mu}}\left(\cos\left(\theta\right)\right) \frac{\der}{\der \theta} \left(P_\xi^{-\mu}\left(\cos\left(\theta\right)\right)\right) - P_\xi^{-\mu}\left(\cos\left(\theta\right)\right) \frac{\der}{\der \theta} \left(P_\xi^{-\Tilde{\mu}}\left(\cos\left(\theta\right)\right)\right)\right)\right]_{\theta=0}^{\pi}\\
    &= \left(\Tilde{\mu}^2-\mu^2\right)\int_{0}^\pi \frac{P_\xi^{-\mu}\left(\cos\left(\theta\right)\right) P_\xi^{-\Tilde{\mu}}\left(\cos\left(\theta\right)\right)}{\sin\left(\theta\right)} \der \theta .
\end{align*}
Using Lemma \ref{lem:behsing} and the assumption $\xi-\mu, \xi-\Tilde{\mu} \in \N \cup \{0\}$ we conclude that the boundary terms vanish. Since $\Tilde{\mu}\neq \mu$ the Legendre functions $P_\xi^{-\mu}$ and $P_\xi^{-\Tilde{\mu}}$ are orthogonal with respect to the weight $\frac{1}{\sin(\theta)}$.

In order to calculate the normalization constant, we abbreviate
\begin{align*}
    I_{\xi,\mu} \coloneqq\int_{0}^\pi \frac{\left(P_\xi^{-\mu}\left(\cos\left(\theta\right)\right)\right)^2}{\sin\left(\theta\right)} \der \theta .
\end{align*}
In the following we will show that $I_{\xi,\mu}$ satisfies the recurrence relation
\begin{align}\label{eq:recrel}
    I_{\xi,\mu-1}=\frac{2\mu}{\mu-1}\left(\xi\left(\xi+1\right)-\mu^2\right)I_{\xi,\mu}-\frac{\mu+1}{\mu-1}\left(\xi-\mu\right)^2\left(\xi+\mu+1\right)^2 I_{\xi,\mu+1} .
\end{align}

From the recurrence formula \eqref{eq:rekc} we conclude
\begin{align*}
    \frac{\left(P_\xi^{-\mu+1}\left(\cos\left(\theta\right)\right)\right)^2}{\sin\left(\theta\right)}= \left( \frac{1}{\sqrt{\sin\left(\theta\right)}}\frac{\der}{\der \theta}\left(P_\xi^{-\mu}\left(\cos\left(\theta\right)\right)\right)+\frac{\mu \cos\left(\theta\right)}{\left(\sin\left(\theta\right)\right)^\frac{3}{2}}P_\xi^{-\mu}\left(\cos\left(\theta\right)\right) \right)^2 .
\end{align*}
Then by integration we have
\begin{align*}
    I_{\xi,\mu-1}=&\int_0^\pi \left( \frac{1}{\sqrt{\sin\left(\theta\right)}}\frac{\der}{\der \theta}\left(P_\xi^{-\mu}\left(\cos\left(\theta\right)\right)\right)\right)^2 + 2 \frac{\mu \cos\left(\theta\right)}{\sin^2\left(\theta\right)} P_\xi^{-\mu}\left(\cos\left(\theta\right)\right) \frac{\der}{\der \theta}\left(P_\xi^{-\mu}\left(\cos\left(\theta\right)\right)\right) \\
    &+\left( \frac{\mu \cos\left(\theta\right)}{\left(\sin\left(\theta\right)\right)^\frac{3}{2}}P_\xi^{-\mu}\left(\cos\left(\theta\right)\right)\right)^2 \der \theta .
\end{align*}
For the first integral, we use integration by parts and the fact that $P_\xi^{-\mu}\left(\cos\left(\theta\right)\right)$ solves \eqref{eq:legeq}.
Hence
\begin{align*}
    \int_0^\pi \frac{1}{\sin\left(\theta\right)} \left( \frac{\der}{\der \theta}\left(P_\xi^{-\mu}\left(\cos\left(\theta\right)\right) \right)\right)^2 \der \theta &= \int_0^\pi 2 \frac{\cos\left(\theta\right)}{\sin^2\left(\theta\right)}  P_\xi^{-\mu}\left(\cos\left(\theta\right)\right) \frac{\der}{\der \theta}\left(P_\xi^{-\mu}\left(\cos\left(\theta\right)\right)\right) \der \theta \\
    &-\int_0^\pi \frac{\mu^2}{\sin^3\left(\theta\right)}\left(P_\xi^{-\mu}\left(\cos\left(\theta\right)\right)\right)^2 \der \theta + \xi\left(\xi+1\right)I_{\xi,\mu} .
\end{align*}
Note that the boundary terms vanish by the same arguments as above and we obtain

\begin{align*}
    I_{\xi,\mu-1}=\left(\xi\left(\xi+1\right)-\mu^2\right)I_{\xi,\mu}+ \left(\mu+1\right) \int_0^\pi \frac{2 \cos\left(\theta\right)}{\sin^2\left(\theta\right)} P_\xi^{-\mu}\left(\cos\left(\theta\right)\right) \frac{\der}{\der \theta}\left(P_\xi^{-\mu}\left(\cos\left(\theta\right)\right)\right) \der \theta .
\end{align*}
Now we use the recurrence formulas \eqref{eq:rekrek},\eqref{eq:rekb} and the orthogonality property from the beginning of the proof to calculate
\begin{align*}
    \int_0^\pi \frac{2 \cos\left(\theta\right)}{\sin^2\left(\theta\right)} P_\xi^{-\mu}\left(\cos\left(\theta\right)\right) \frac{\der}{\der \theta}\left(P_\xi^{-\mu}\left(\cos\left(\theta\right)\right)\right) \der \theta =\frac{1}{2\mu} I_{\xi,\mu-1}-\frac{1}{2\mu}\left(\xi+\mu+1\right)^2\left(\xi-\mu\right)^2 I_{\xi,\mu+1} .
\end{align*}
Altogether we have
\begin{align*}
    I_{\xi,\mu-1}=\frac{\mu+1}{2\mu}I_{\xi,\mu-1}-\frac{\mu+1}{2\mu}\left(\xi-\mu\right)^2\left(\xi+\mu+1\right)^2 I_{\xi,\mu+1}+\left(\xi\left(\xi+1\right)-\mu^2\right)I_{\xi,\mu}
\end{align*}
which is equivalent to \eqref{eq:recrel}. \\
One can show that $\frac{1}{\mu}\frac{\Gamma\left(\xi-\mu+1\right)}{\Gamma\left(\xi+\mu +1\right)} $ also satisfies the recurrence relation \eqref{eq:recrel} and that the initial values
\begin{align}\label{eq:initval1}
    I_{L+\frac{1}{2}, \frac{1}{2}}= \int_{0}^\pi \frac{\left(P_{L+\frac{1}{2}}^{-\frac{1}{2}}\left(\cos\left(\theta\right)\right)\right)^2}{\sin\left(\theta\right)} \der \theta = \frac{2}{L+1}=\frac{2 \Gamma\left(L+1\right)}{\Gamma\left(L+2\right)}
\end{align}
\begin{align}\label{eq:initval2}
     I_{L+\frac{1}{2}, \frac{3}{2}}= \int_{0}^\pi \frac{\left(P_{L+\frac{1}{2}}^{-\frac{3}{2}}\left(\cos\left(\theta\right)\right)\right)^2}{\sin\left(\theta\right)} \der \theta =\frac{2}{3}\frac{1}{\left(L+2\right)\left(L+1\right)L}=\frac{2}{3} \frac{\Gamma\left(L\right)}{\Gamma\left(L+3\right)}
\end{align}
for $L\geq 1$ coincide. This then provides the formula of the normalization constant. \\
We sketch the proof of the identities \eqref{eq:initval1} and \eqref{eq:initval2}.
From the formula \eqref{eq:mueinhalb} we conclude
\begin{align*}
    I_{L+\frac{1}{2}, \frac{1}{2}}=  \frac{2}{\pi (L+1)^2} \int_0^\pi \frac{\left( \sin\left(\left( L+1\right)\theta \right) \right)^2}{\left( \sin\left(\theta \right) \right)^2} \der \theta .
\end{align*}
To calculate the integral we use the identity
\begin{align}\label{eq:sinsquare}
    \frac{\left( \sin\left(\left( L+1\right)\theta \right) \right)^2}{\left( \sin\left(\theta \right) \right)^2}= L+1+2 \sum_{k=1}^{L}\left(L+1-k\right)\cos\left(2k\theta\right)
\end{align}
which follows from Euler's formula, the formula for the finite geometric sum and a suitable rearrangement of the terms. This then verifies the identity \eqref{eq:initval1}. 
In order to show the identity \eqref{eq:initval2} we first use \eqref{eq:rec} and obtain
\begin{align*}
     I_{L+\frac{1}{2}, \frac{3}{2}}&= \int_{0}^\pi \frac{\left(P_{L+\frac{1}{2}}^{-\frac{3}{2}}\left(\cos\left(\theta\right)\right)\right)^2}{\sin\left(\theta\right)} \der \theta  \\
     &= \frac{1}{9} \int_0^\pi \sin\left( \theta \right) \left( P_{L+\frac{3}{2}}^{-\frac{1}{2}}\left( \cos \left( \theta\right) \right) \right)^2+ \left(L+3\right)^2\left(L+4\right)^2  \sin\left( \theta \right) \left( P_{L+\frac{3}{2}}^{-\frac{5}{2}}\left( \cos \left( \theta\right) \right) \right)^2 \\
     & \quad \quad \quad  +2 \left( L+3\right)\left(L+4 \right) \sin \left(\theta\right) P_{L+\frac{3}{2}}^{-\frac{1}{2}}\left( \cos \left( \theta\right) \right) P_{L+\frac{3}{2}}^{-\frac{5}{2}}\left( \cos \left( \theta\right) \right) \der \theta .
\end{align*}
The first two integrals containing the square of a Legendre function are given by formula (2.11) from \cite{higuchi}.
For the last integral we use \eqref{eq:reccc} twice to increase the order of the Legendre function $P_{L+\frac{3}{2}}^{-\frac{5}{2}}$ from $-\frac{5}{2}$ to $-\frac{1}{2}$, i.e.
\begin{align*}
   & \int_0^\pi \sin \left(\theta\right) P_{L+\frac{3}{2}}^{-\frac{1}{2}}\left( \cos \left( \theta\right) \right) P_{L+\frac{3}{2}}^{-\frac{5}{2}}\left( \cos \left( \theta\right) \right) \der \theta\\
   &= \frac{1}{4(L+2)(L+3)}  \int_0^\pi \frac{1}{\sin\left(\theta \right)} P_{L+\frac{3}{2}}^{-\frac{1}{2}}\left( \cos \left( \theta\right) \right) P_{L+\frac{7}{2}}^{-\frac{1}{2}}\left( \cos \left( \theta\right) \right) \der \theta   \\
   &+ \frac{1}{4(L+1)(L+2)}  \int_0^\pi \frac{1}{\sin\left(\theta \right)} P_{L+\frac{3}{2}}^{-\frac{1}{2}}\left( \cos \left( \theta\right) \right) P_{L-\frac{1}{2}}^{-\frac{1}{2}}\left( \cos \left( \theta\right) \right) \der \theta \\
   &-\frac{1}{2(L+1)(L+3)} \int_0^\pi \frac{1}{\sin\left(\theta \right)}\left( P_{L+\frac{3}{2}}^{-\frac{1}{2}}\left( \cos \left( \theta\right) \right) \right)^2  .  
\end{align*}
The last integral can be calculated using the identity \eqref{eq:initval1} for $L+1$.
The first integral can be calculated by making use of formula \eqref{eq:mueinhalb}, trigonometric addition theorems and the identity \eqref{eq:sinsquare}. The second integral then follows from the first integral by replacing $L$ by $L-2$. 
Merging all the calculations together leads to the identity \eqref{eq:initval2}.
\end{proof}

In case of integer order and integer degree we have the following result to bound the Legendre functions.
\begin{lemma}\label{lem:boundlohöfer} \label{lem:legpolybound}
    For any $m,n \in \N$ with $1 \leq m \leq n$, we have 
\begin{align*}
    \max_{s \in [-1,1]}|P_n^m\left(s\right)| \left( \frac{\left(n-m\right)!}{\left(n+m\right)!} \right)^\frac{1}{2} < \frac{2}{m^\frac{1}{4}}.
\end{align*}
For any $n \in \N$ we have 
  \begin{align*}
        \max_{s\in [-1,1] }|P_n^0\left(s\right)| \leq 1.
    \end{align*}
\end{lemma}
\begin{proof}
    The first result goes back to Lohöfer, Corollary 3 in \cite{lohöfer}. The second result can be found in Chapter 10 in \cite{htf}.
\end{proof}
The following result verifies \eqref{eq:legbound} for the special case $i=2$.
\begin{lemma}\label{lem:i2}
    There exists a constant $c>0$ such that
   \begin{align*}
     \sup_{\theta\in[0,\pi]}|{_{2}}{\overline{P}}_{L}^{K}\left(\theta\right)| \leq c \left(L+1\right)^\frac{1}{2}
\end{align*}
   for all $0\leq |K|\leq L$
\end{lemma}
\begin{proof}
    We have
    \begin{align*}
         {_{2}}{\overline{P}}_{L}^{K}\left(\theta\right)=\sqrt{\frac{2L+1}{2}\frac{\left(L+K\right)!}{\left(L-K\right)!}}P_L^{-K}\left(\cos\left(\theta\right)\right)
    \end{align*}
    and for the case $0=K\leq L$, Lemma \ref{lem:legpolybound} directly implies 
    \begin{align*}
         \sup_{\theta\in[0,\pi]}|{_{2}}{\overline{P}}_{L}^{0}\left(\theta\right)| \leq \sqrt{\frac{2L+1}{2}} \leq c \left(L+1\right)^\frac{1}{2}
    \end{align*}
   for a constant $c>0$ independent of $L$.  Next, for $-L\leq K \leq -1$ we use Lemma \ref{lem:boundlohöfer} and obtain
   \begin{align*}
       \sup_{\theta\in[0,\pi]}|{_{2}}{\overline{P}}_{L}^{-K}\left(\theta\right)| \leq \frac{1}{\left(-K\right)^\frac{1}{4}}\sqrt{\frac{2L+1}{2}} \leq c \left(L+1\right)^\frac{1}{2}
   \end{align*}
   for a constant $c>0$ independent of $L$ and $K$. Lastly, for the case $1\leq K\leq L$ we use formula \eqref{eq:negativeind} and then proceed as before using Lemma \ref{lem:boundlohöfer}.
\end{proof}

Similarly we can prove \eqref{eq:legbound} for the special case $i=4$ directly.
\begin{lemma}\label{lem:i4}
There exists a constant $c >0$ such that
 \begin{align*}
         \sup_{\theta\in[0,\pi]}|{_{4}}{\overline{P}}_{L}^{K}\left(\theta\right)| < c \left(L+1\right)^{\frac{3}{2}}
    \end{align*}
     for all $0 \leq K \leq L$.
\end{lemma}
\begin{proof}
First, using \eqref{eq:negativeind} and then applying the recurrence relation \eqref{eq:rec} we obtain
    \begin{align*}
        {_{4}}{\overline{P}}_{L}^{K}\left(\theta\right) &= \sqrt{\frac{2L+3}{2} \frac{\left(L+K+2\right)!}{\left(L-K\right)!}}\sin\left(\theta\right)^{-1}P_{L+1}^{-\left(K+1\right)}\left(\cos\left(\theta\right)\right)\\
        &=\left(-1\right)^{K+1} \sqrt{\frac{2L+3}{2}\frac{\left(L-K\right)!}{\left(L+K+2\right)!}} \sin\left(\theta\right)^{-1}P_{L+1}^{K+1}\left(\cos\left(\theta\right)\right) \\
        &=\frac{\left(-1\right)^K}{2K+2} \sqrt{\frac{2L+3}{2}\frac{\left(L-K\right)!}{\left(L+K+2\right)!}}  \left( P_{L+2}^{K+2}\left(\cos\left(\theta\right)\right)+ \left(L-K+1\right)\left(L-K+2\right) P_{L+2}^K\left(\cos\left(\theta\right)\right)  \right).
    \end{align*}
  Let $0=K\leq L$. Then
  \begin{align*}
       {_{4}}{\overline{P}}_{L}^{0}\left(\theta\right) &=\frac{1}{2} \sqrt{\frac{2L+3}{2\left(L+2\right)\left(L+1\right)}} \left( P_{L+2}^2\left(\cos\left(\theta\right)\right)+\left(L+1\right)\left(L+2\right)P_{L+2}^0\left(\cos\left(\theta\right) \right) \right) 
  \end{align*}
  and from Lemma \ref{lem:boundlohöfer} and \ref{lem:legpolybound} we conclude
  \begin{align*}
      \sup_{\theta\in[0,\pi]}|{_{4}}{\overline{P}}_{L}^{0}\left(\theta\right)| &\leq \frac{1}{2} \sqrt{\frac{2L+3}{2\left(L+2\right)\left(L+1\right)}} \left( 2^{3/4} \sqrt{\left(L+4\right)\left(L+3\right)\left(L+2\right)\left(L+1\right)}+\left(L+1\right)\left(L+2\right) \right) \\
      &\leq c \left(L+1\right)^{\frac{3}{2}}
  \end{align*}
  for a constant $c>0$ independent of $L$. Similarly, for $1\leq K \leq L$ we estimate by Lemma \ref{lem:boundlohöfer}
  \begin{align*}
      \sup_{\theta\in[0,\pi]}|{_{4}}{\overline{P}}_{L}^{K}\left(\theta\right)|  &\leq \frac{1}{2K+2} \sqrt{\frac{2L+3}{2}\frac{\left(L-K\right)!}{\left(L+K+2\right)!}} & &\bigg( \frac{1}{\left(K+2\right)^\frac{1}{4}}\sqrt{\frac{\left(L+K+4\right)!}{\left(L-K\right)!}} \\
    & & &+ \frac{\left(L-K+1\right)\left(L-K+2\right)}{K^\frac{1}{4}} \sqrt{\frac{\left(L+K+2\right)!}{\left(L-K+2\right)!}}  \bigg) \\
    &\leq c \left(L+1\right)^\frac{3}{2}
  \end{align*}
  for a constant $c>0$ independent of $L$ and $K$.  
\end{proof}

We also prove estimate \eqref{eq:legbound} for the special case $i=3$.

\begin{lemma}\label{lem:i3}
    There exists a constant $c>0$ such that
    \begin{align}\label{eq:esti3}
         \sup_{\theta\in[0,\pi]}|{_{3}}{\overline{P}}_{L}^{K}\left(\theta\right)| < c \left(L+1\right)^\frac{3}{2}
    \end{align}
     for all $0 \leq K \leq L$.
\end{lemma}

\begin{proof}
    We have 
    \begin{align*}
         {_{3}}{\overline{P}}_{L}^{K}\left(\theta\right)= \sqrt{\left(L+1\right) \frac{\left(L+K+1\right)!}{\left(L-K\right)!}}\frac{1}{\sqrt{\sin\left(\theta\right)}}P_{L+\frac{1}{2}}^{-\left(K+\frac{1}{2}\right)}\left(\cos\left(\theta\right)\right).
    \end{align*}
First, if $K=0$ and $L \geq 0$ we have by \eqref{eq:mueinhalb}
    \begin{align*}
        P_{L+\frac{1}{2}}^{-\frac{1}{2}}\left(\cos\left(\theta\right)\right)=\sqrt{\frac{2}{\pi}} \frac{\sin\left(\left(L+1\right) \theta\right)}{\left(L+1\right) \sqrt{\sin\left(\theta\right)}}
    \end{align*}
    and in particular
    \begin{align*}
         \sup_{\theta \in [0,\pi]}  |{_{3}}{\overline{P}}_{L}^{0}\left(\theta\right)|= \sup_{\theta \in [0,\pi]}  \left|\sqrt{\frac{2}{\pi}} \frac{\sin\left(\left(L+1\right)\theta\right)}{\sin\left(\theta\right)}\right| \leq \sqrt{\frac{2}{\pi}} \left(L+1\right) .
    \end{align*}
Next, let $1\leq K \leq L$. We use Lemma \ref{lem:behsing} and the fact that $L-K \in \N \cup \{0\}$ and observe that   
\begin{align*}
    \lim_{\theta \to 0 }{_{3}}{\overline{P}}_{L}^{K}\left(\theta\right)=0, \quad \lim_{\theta \to \pi }{_{3}}{\overline{P}}_{L}^{K}\left(\theta\right)=0
\end{align*}
for $1\leq K \leq L$. 
In the following we will show that $ \left\| \frac{\mathrm{d}}{\mathrm{d}\theta} \left(  {_{3}}{\overline{P}}_{L}^{K}\left(\theta\right)\right) \right\|_{L^2\left(0,\pi\right)}^2 \leq c \left(L+1\right)^3$.
Then, by Sobolev's embedding theorem and Poincarè's inequality we obtain the bound as stated in \eqref{eq:esti3}.
For readability we omit the prefactor of ${_{3}}{\overline{P}}_{L}^{K}$ depending on $K$ and $L$ in the following calculations.
By the product rule we obtain
\begin{align*}
    &\int_0^\pi \left( \frac{\mathrm{d}}{\mathrm{d}\theta} \left( \frac{1}{\sqrt{\sin\left(\theta\right)}}P_{L+\frac{1}{2}}^{-\left(K+\frac{1}{2}\right)}\left(\cos\left(\theta\right)\right) \right) \right)^2 \der \theta \\
    = &\int_0^\pi \left( \frac{\mathrm{d}}{\mathrm{d}\theta} \frac{1}{\sqrt{\sin\left(\theta\right)}}\right)^2 \left( P_{L+\frac{1}{2}}^{-\left(K+\frac{1}{2}\right)}\left(\cos\left(\theta\right)\right)\right)^2 + \left(  \frac{1}{\sqrt{\sin\left(\theta\right)}}\right)^2 \left(\frac{\mathrm{d}}{\mathrm{d}\theta}\left( P_{L+\frac{1}{2}}^{-\left(K+\frac{1}{2}\right)}\left(\cos\left(\theta\right)\right)\right)\right)^2 \\
    &+ 2 \frac{1}{\sqrt{\sin\left(\theta\right)}} P_{L+\frac{1}{2}}^{-\left(K+\frac{1}{2}\right)}\left(\cos\left(\theta\right)\right) \left(\frac{\mathrm{d}}{\mathrm{d}\theta} \frac{1}{\sqrt{\sin\left(\theta\right)}} \right)\left( \frac{\mathrm{d}}{\mathrm{d}\theta}\left( P_{L+\frac{1}{2}}^{-\left(K+\frac{1}{2}\right)}\left(\cos\left(\theta\right)\right)\right)\right) \der \theta .
\end{align*}
We calculate the three integrals separately.
For the first integral we use \ref{eq:rekrek} and then Lemma~\ref{lem:behsing}
\begin{align*}
    &\int_0^\pi \left( \frac{\mathrm{d}}{\mathrm{d}\theta}  \frac{1}{\sqrt{\sin\left(\theta \right)}}\right) \left(P_{L+\frac{1}{2}}^{-\left(K+\frac{1}{2}\right)}\left(\cos\left(\theta\right)\right)  \right)^2 \der \theta \\
    &=\int_0^\pi \frac{1}{4 \sin\left(\theta\right)} \left( \frac{\cos\left(\theta\right)}{\sin\left(\theta\right)}P_{L+\frac{1}{2}}^{-\left(K+\frac{1}{2}\right)}\left(\cos\left(\theta\right)\right) \right)^2 \der \theta \\
    &=\int_0^\pi \frac{1}{4 \sin\left(\theta\right)} \frac{1}{\left(2K+1\right)^2}\left(P_{L+\frac{1}{2}}^{-\left(K-\frac{1}{2}\right)}\left(\cos\left(\theta\right)\right)+ \left(L+K+2\right)\left(L-K\right) P_{L+\frac{1}{2}}^{-\left(K+\frac{3}{2}\right)}\left(\cos\left(\theta\right)\right)  \right)^2 \der \theta \\
    &=\frac{1}{2\left(2K+1\right)^2} \frac{1}{2K-1} \frac{\left(L-K+1\right)!}{\left(L+K\right)!}+ \frac{1}{2\left(2K+1\right)^2} \frac{1}{2K+3} \frac{\left(L-K\right)!}{\left(L+K+1\right)!} \left(L+K+2\right)\left(L-K\right) .
\end{align*}
For the second integral we use \eqref{eq:rekb} and Lemma ~ref{lem:behsing}
\begin{align*}
    &\int_0^\pi \frac{1}{\sin\left(\theta\right)} \left(\frac{\mathrm{d}}{\mathrm{d}\theta}\left( P_{L+\frac{1}{2}}^{-\left(K+\frac{1}{2}\right)}\left(\cos\left(\theta\right)\right)\right)\right)^2 \der \theta \\
    &=\int_0^\pi \frac{1}{\sin\left(\theta\right)} \frac{1}{4} \left(\left(L-K\right)\left(L+K+2\right)P_{L+\frac{1}{2}}^{-\left(K+\frac{3}{2}\right)}\left(\cos\left(\theta\right)\right) -P_{L+\frac{1}{2}}^{-\left(K-\frac{1}{2}\right)}\left(\cos\left(\theta\right)\right) \right)^2 \der \theta \\
    &=\frac{1}{2} \left(L-K\right)\left(L+K+2\right) \frac{1}{2K+3} \frac{\left(L-K\right)!}{\left(L+K+1\right)!}+ \frac{1}{2} \frac{1}{2K-1} \frac{\left(L-K+1\right)!}{\left(L+K\right)!} .
\end{align*}
For the last integral we use \eqref{eq:rekrek}, \eqref{eq:rekb} and Lemma \ref{lem:behsing}
\begin{align*}
  &\int_0^\pi  2 \frac{1}{\sqrt{\sin\left(\theta\right)}} P_{L+\frac{1}{2}}^{-\left(K+\frac{1}{2}\right)}\left(\cos\left(\theta\right)\right) \left(\frac{\mathrm{d}}{\mathrm{d}\theta} \frac{1}{\sqrt{\sin\left(\theta\right)}} \right)\left( \frac{\mathrm{d}}{\mathrm{d}\theta}\left( P_{L+\frac{1}{2}}^{-\left(K+\frac{1}{2}\right)}\left(\cos\left(\theta\right)\right)\right)\right) \der \theta \\
  &= \int_0^\pi\frac{-\cos\left(\theta\right)}{\sin^2\left(\theta\right)} P_{L+\frac{1}{2}}^{-\left(K+\frac{1}{2}\right)}\left(\cos\left(\theta\right)\right) \left( \frac{\mathrm{d}}{\mathrm{d}\theta}\left( P_{L+\frac{1}{2}}^{-\left(K+\frac{1}{2}\right)}\left(\cos\left(\theta\right)\right)\right)\right)  \der \theta \\
  &=\int_0^\pi \frac{-1}{2(2K+1) \sin\left(\theta\right) } \left( \left(P_{L+\frac{1}{2}}^{-\left(K-\frac{1}{2}\right)}\left(\cos\left(\theta\right)\right) \right)^2-\left(L-K\right)^2\left(L+K+2\right)^2\left(P_{L+\frac{1}{2}}^{-\left(K+\frac{3}{2}\right)}\left(\cos\left(\theta\right)\right)\right)^2 \right) \der \theta\\
   &=-\frac{1}{2K+1} \frac{1}{2K-1} \frac{\left(L-K+1\right)!}{\left(L+K\right)!} + \frac{1}{2K+1}\frac{1}{2K+3}\left(L-K\right)\left(L+K+2\right) \frac{\left(L-K\right)!}{\left(L+K+1\right)!} .
\end{align*}
   Altogether, including the prefactor of ${_{3}}{\overline{P}}_{L}^{K}$, we have
   \begin{align*}
       &\left\| \frac{\mathrm{d}}{\mathrm{d}\theta} \left(  {_{3}}{\overline{P}}_{L}^{K}\left(\theta\right)\right) \right\|_{L^2\left(0,\pi\right)}^2\\
       &=\frac{\left(L+1\right)}{\left(2K+1\right)^2}\left( \left(L-K+1\right)\left(L+K+1\right)\frac{2K^2}{2K-1}+\left(L-K\right)\left(L+K+2\right) \frac{2\left(K+1\right)^2}{2K+3} \right) \\
       &\leq c \left(L+1\right)^3
   \end{align*}
   as claimed.
   \end{proof}
\begin{proof}[Proof of Theorem \ref{lem:boundsh}]
    We prove the statement via induction over $i$ and distinguish the two cases $i$ even and $i$ odd in order to use recurrence formulas for the Legendre functions. We present the proof for the case when $i$ is even. The proof in the case when $i$ is odd works the same way with Lemma \ref{lem:i3} as the base case for the induction. 
    In the following, let $i$ be even: For $i=2$ and $i=4$ the statement holds true by Lemma \ref{lem:i2} and Lemma \ref{lem:i4} respectively.\\
    \textit{Induction hypothesis:} For $i \geq 2$, even, arbitrary but fixed there exists $c>0$ such that
    \begin{align*}
          \sup_{\theta \in [0,\pi]}|{_{i}}{\overline{P}}_{L}^{K}\left(\theta\right) | \leq c \left(L+1\right)^{\frac{i-1}{2}} 
    \end{align*}
    for all $0\leq K \leq L$.\\
    \textit{Inductive step:}  $i \to i+2$: \\
    We have
    \begin{align*}
         {_{i+2}}{\overline{P}}_{L}^{K}\left(\theta\right)=\sqrt{\frac{2L+i+1}{2} \frac{\left(L+K+i\right)!}{\left(L-K\right)!}} \left(\sin\left(\theta\right)\right)^{\frac{2-i}{2}} \left(\sin\left(\theta\right)\right)^{-1} P_{L+1+\frac{i-2}{2}}^{-\left(K+1+\frac{i-2}{2}\right)}\left(\cos\left(\theta\right)\right)
    \end{align*}
    and by \eqref{eq:rec} we obtain 
    \begin{align*}
        &\left(\sin\left(\theta\right)\right)^{-1} P_{L+1+\frac{i-2}{2}}^{-\left(K+1+\frac{i-2}{2}\right)}\left(\cos\left(\theta\right)\right) \\
        &=\frac{1}{2K+i} \left(P_{L+2+\frac{i-2}{2}}^{-\left(K+\frac{i-2}{2}\right)}\left(\cos\left(\theta\right) \right)+ \left(L+K+1+i\right)\left(L+K+2+i\right) P_{L+2+\frac{i-2}{2}}^{-\left(K+2+\frac{i-2}{2}\right)}\left(\cos\left(\theta\right)\right) \right).
    \end{align*}
We estimate the two terms separately and use the inductive hypothesis both times. In the following, $c>0$ is a constant only depending on $i$, which may vary from line to line. First we have
\begin{align*}
    &\left| \sqrt{\frac{2L+i+1}{2} \frac{\left(L+K+i\right)!}{\left(L-K\right)!}}\left(\sin\left(\theta\right)\right)^{\frac{2-i}{2}} \frac{1}{2K+i} P_{L+2+\frac{i-2}{2}}^{-\left(K+\frac{i-2}{2}\right)}\left(\cos\left(\theta\right) \right) \right| \\
    &= \left| {_{i}}{\overline{P}}_{L+2}^{K}\left(\theta\right) \frac{1}{2K+i} \sqrt{\frac{2L+1+i}{2L+3+i}(L+2-K)(L+1-K)} \right| \\
    &\leq c(L+3)^{\frac{i-1}{2}} \left| \frac{1}{2K+i} \sqrt{\frac{2L+1+i}{2L+3+i}(L+2-K)(L+1-K)} \right| \\
    &\leq c(L+3)^{\frac{i-1}{2}}(L+3)\\
    &\leq c(L+1)^{\frac{(i+2)-1}{2}} 
\end{align*}
and secondly
\begin{align*}
  &\left|\sqrt{\frac{2L+i+1}{2} \frac{\left(L+K+i\right)!}{\left(L-K\right)!}}\left(\sin\left(\theta\right)\right)^{\frac{2-i}{2}} \frac{1}{2K+i}  \left(L+K+1+i\right)\left(L+K+2+i\right) P_{L+2+\frac{i-2}{2}}^{-\left(K+2+\frac{i-2}{2}\right)}\left(\cos\left(\theta\right)\right)\right| \\
  &= \left| {_{i}}{\overline{P}}_{L+2}^{K+2}\left(\theta\right) \frac{1}{2K+i} \sqrt{\frac{2L+1+i}{2L+3+i}(L+K+1+i)(L+K+2+i)} \right| \\
  &\leq c (L+3)^{\frac{i-1}{2}} \left|\frac{1}{2K+i} \sqrt{\frac{2L+1+i}{2L+3+i}(L+K+1+i)(L+K+2+i)} \right| \\
  &\leq c (L+3)^{\frac{i-1}{2}} (L+3)\\
  &\leq c (L+1)^{\frac{(i+2)-1}{2}}.
  \end{align*}
Altogether we have
\begin{align*}
   | {_{i+2}}{\overline{P}}_{L}^{K}\left(\theta\right) | \leq c \left(L+1\right)^{\frac{i+1}{2}}
\end{align*}
where $c$ only depends on $i$. Since $1\leq i \leq N$ and $N$ is fixed, this finishes the proof of Lemma~\ref{lem:boundsh}.
\end{proof}
Altogether this concludes the proof of Example \ref{ex:wl}.

Lastly we show that the Examples \ref{lem:multistepcond}--\ref{lem:interface} satisfy the conditions \ref{as:dgen}--\ref{as:last}. 
First, we observe that assumptions \ref{as:dgen} and \ref{as:dper} are already satisfied by definition for all three Examples \ref{lem:multistepcond}--\ref{lem:interface}.
Note that in the case of the purely periodic potential, Example \ref{lem:multistepcond}, condition \ref{as:dpointspec} holds trivially since $\sigma_p(- \tfrac{1}{d(t)} \dtsquare)=\emptyset$ by Theorem 5.3.1 in \cite{Eastham}.
The eigenvalues in Example \ref{ex:qho} and \ref{ex:wl} are of the form $\nu_n=(2n+1)^2$ and $\nu_n=(2n+N-1)^2$, respectively, for $n \in \No$.
Therefore, Example \ref{lem:multistepcond}, Example \ref{ex:dislint} and Example \ref{lem:interface} satisfy the assumptions \ref{as:dpointspec}--\ref{as:last}, which follows from Lemma C.1, Lemma C.6 and Lemma C.7 in \cite{hor}, respectively, by choosing $\omega=1$ therein.

    \section{Proofs of embedding results} \label{sec:approof}

In this section we prove the embedding results from Theorem \ref{theom:emb} and \ref{lem:cct}.
We begin with a useful estimate.

\begin{lemma}\label{lem:est}
    There exists $c>0$ such that 
\begin{align*}
    b_{|L_n|}(v,v)=\norm{v}_{\Hs_n}^2\geq c( \sqrt{\nu_n} \norm{v}_{L^2_d(\R)}^2+\frac{1}{\sqrt{\nu_n}} \norm{v'}_{L^2(\R)}^2)
\end{align*}
for every $n \in \N$ and $v \in H^1(\R)$.
\end{lemma}
\begin{proof}
    By \eqref{eq:specest} we have
\begin{align}\label{eq:estnorm}
    b_{|L_n|}(v,v)= \norm{v}_{\Hs_n}^2 \geq c \sqrt{\nu_n}\norm{v}_{L^2_d(\R)}^2
\end{align}
for all $n \in \N$ and $v \in H^1(\R)$.

Now let $v \in \Hs_n^+$, then
\begin{align*}
 0 \leq     b_{|L_n|}(v,v)=b_{L_n}(v,v) = \int_\R - |v'|^2 + \nu_n d(t) |v|^2 \dt 
\end{align*}
and hence
\begin{align*}
    \nu_n \norm{v}_{L^2_d(\R)}^2 \geq \norm{v'}_{L^2(\R)}^2
\end{align*}
which implies together with \eqref{eq:estnorm}
\begin{align*}
    b_{L_n}(v,v)\geq \sqrt{\nu_n}\norm{v}_{L^2_d(\R)}^2\geq \frac{1}{\sqrt{\nu_n}}\norm{v'}_{L^2(\R)}^2.
\end{align*}
For $v \in \Hs_n^-$ we have together with \eqref{eq:estnorm}
\begin{align*}
    b_{|L_n|}(v,v)=b_{-L_n}(v,v)= \int_\R  |v'|^2 - \nu_n d(t) |v|^2 \dt \geq \sqrt{\nu_n}\norm{v}_{L^2_d(\R)}^2
\end{align*}
which implies
\begin{align*}
    \frac{\nu_n}{\nu_n + \sqrt{\nu_n}}\norm{v'}_{L^2(\R)}^2 \geq \nu_n \norm{v}_{L^2_d(\R)}^2
\end{align*}
and therefore
\begin{align*}
    b_{-L_n}(v,v)= \int_\R  |v'|^2 - \nu_n d(t) |v|^2 \dt \geq \int_\R  |v'|^2 - \frac{\nu_n}{\nu_n + \sqrt{\nu_n}}  |v'|^2 \dt \geq c \frac{1}{\sqrt{\nu_n}} \norm{v'}_{L^2(\R)}^2 .  
    &\qedhere
\end{align*}
\end{proof}

\begin{proof}[Proof of Theorem \ref{theom:emb}]
The continuity and local compactness of $\Hs \hookrightarrow L^q_{Vd}(\R^N\times \R)$ follows from Theorem 3.27 in \cite{hor}. Note that in \cite{hor} the eigenvalues of the operator with purely discrete spectrum are counted with respect to their multiplicity. For our purposes we do not count the eigenvalues of $\frac{1}{V}M$ with respect to their multiplicity, instead we introduced the dimension $d_n$ of the eigenspaces to the eigenvalues $\nu_n$ and adapted the assumption from Theorem 3.27 in \cite{hor} resulting in assumption \ref{as:efestimateM}.
It remains to show the compactness of the embedding global in space and local in time. Let $\tau >0$.
We show that $\id: \Hs \hookrightarrow L^2_V(\R^N \times [-\tau, \tau])$ is compact.
This then implies, together with Hölder's inequality and the boundedness of the embedding $H \hookrightarrow L_{Vd}^q(\R^N\times [-\tau,\tau])$, the compactness of the embedding for the remaining cases of $q$.

For $K \in \N$ fixed we define the space
\begin{align*}
    \Hs^K \coloneqq \{ u (x,t)\in L^2_{Vd}(\R^N\times \R) : u(x,t)=\sum_{n=1}^K \sum_{l=1}^{d_n} u_{n,l}(t) \varphi_{n,l}(x) \text{ and }  \norm{u}_{\Hs^K}< \infty \} ,
\end{align*}
where
\begin{align*}
    \norm{u}_{\mathcal{H}^K}^2= \sum_{n=1}^K \sum_{l=1}^{d_n}\norm{u_{n,l}}_{\Hs_n}^2
\end{align*}
and the operator 
\begin{align*}
    \id^K : \Hs^K \to L^2_{Vd}(\R^n \times \R), \quad  \id^K u=\sum_{n =1}^K \sum_{l=1}^{d_n} u_{n,l}(t) \varphi_{n,l}(x) .
\end{align*}
We show that $\id^K : \Hs^K \to L^2_{Vd}(\R^n \times [-\tau, \tau] )$ is compact.
Let $(u^j)_{j \in \N}$ be a bounded sequence in $\Hs^K$, meaning $\norm{u^j}_{\Hs^K}\leq c$ for all $j \in \N$.
This implies in particular that $\norm{u_{n,l}^j}_{\Hs_n} \leq c$ for all $1 \leq n \leq  K$, $1\leq l \leq d_n$ and $j \in \N$.
Further by Lemma \ref{lem:est} we have
\begin{align*}
    \norm{u_{n,l}^j}_{H^1(\R)} \leq \sqrt{\nu_K} \norm{u_{n,l}^j}_{\Hs_n}
\end{align*}
and hence $\norm{u_{n,l}^j}_{H^1(\R)}\leq c$ for all $1 \leq n \leq  K$, $1 \leq l \leq d_n$ and $j \in \N$.
By the compactness of the embedding $H^1(\R) \hookrightarrow L^2[-\tau,\tau]$, each sequence $(u_{n,l}^j)_{j \in \N}$, $1 \leq n \leq  K$, $1 \leq l \leq d_n$ has a convergent subsequence in $L^2[-\tau , \tau] $.
Now, we choose a subsequence $(u_{n,l}^{j_m})_{m \in \N}$ of $(u_{n,l}^j)_{j \in \N}$ such that $u_{n,l}^{j_m} \to u_{n,l}$ in $L^2[- \tau, \tau]$ as $m \to \infty$ for all $1 \leq n \leq  K$, $1 \leq l \leq d_n$. 

Hence 
\begin{align*}
    \norm{\sum_{n=1}^K \sum_{l=1}^{d_n} u_{n,l}^{j_m}(t)\varphi_{n,l}(x) -\sum_{n=1}^K \sum_{l=1}^{d_n} u_{n,l}(t)\varphi_{n,l}(x) }_{L^2(\R^N \times [-\tau,\tau])}^2 \to 0
\end{align*}
as $m \to \infty$ which shows the compactness of $\id^K$.

Lastly we show that $\id^K$ converges to $\id:\Hs \to L^2_{Vd}(\R^n \times \R)$ in operator norm and hence $\id$ is compact.
From \eqref{eq:specest} we infer
\begin{align*}
    \norm{(\id^K-\id)u}_{L^2_{Vd}(\R^N \times \R)}^2 =\sum_{n=K+1}^\infty \sum_{l=1}^{d_n} \norm{u_{n,l}}_{L^2_d(\R)}^2 
  &\leq \sum_{n=K+1}^\infty \sum_{l=1}^{d_n} \frac{1}{\sqrt{\nu_{K+1}}} \norm{u_{n,l}}_{\mathcal{H}_n}^2
    &\leq \frac{c}{\sqrt{\nu_{K+1}}} \norm{u}_{\Hs}^2
\end{align*}
where $\nu_{K+1} \to \infty$ as $K \to \infty$.
\end{proof}

The proof of Lemma \ref{lem:cct} is done similarly as in \cite{maier_reichel_schneider}. We define the intermediate spaces 
\begin{align*}
H \coloneqq \{ u(x,t)=\sum_{n \in \N} \sum_{l=1}^{d_n} u_{n,l}(t) \varphi_{n,l}(x) : \| u \|_H^2 \coloneqq \sum_{n \in \N} \sum_{l=1}^{d_n} \sqrt{\nu_n} \|u_{n,l}\|_{L^2_d(\R)}^2 + \frac{1}{\sqrt{\nu_n}} \|u_{n,l}' \|_{L^2(\R)}^2 < \infty \}
\end{align*}
and
\begin{align*}
 X_r \coloneqq \{ u(x,t)= \sum_{n \in \N} \sum_{l=1}^{d_n} u_{n,l}(t) \varphi_{n,l}(x) : \| u \|_{X_r}^{r'}\coloneqq \sum_{n \in \N} \sum_{l=1}^{d_n}  \norm{\frac{u_{n,l}}{s_{n,l}} }_{L^r_d(\R)}^{r'} s_{n,l}^2 < \infty \}
\end{align*}
for $r \in (1, \infty]$ with $\frac{1}{r}+ \frac{1}{r'}=1$ and $s_{n,l} \coloneqq \| \varphi_{n,l}\|_\infty$ for $1 \leq l \leq d_n$.
The space $X_r$ captures the growth of the eigenfunctions $\varphi_{n,l}$ in $n$ according to our assumption \ref{as:efestimateM}.
\begin{lemma}\label{lem:tripemb}
    The three embeddings 
    \begin{align*}
 \Hs  \hookrightarrow  H  \hookrightarrow X_r \hookrightarrow  L^r_V(\R^N \times \R) 
\end{align*}
are continuous for $ r \in [2, \frac{4(1+\alpha)}{2\alpha+1} )$.
\end{lemma}

\begin{proof}
Step 1: The continuity of $\Hs  \hookrightarrow  H$ follows from the estimate in Lemma \ref{lem:est}.

Step 2: We show $H \hookrightarrow X_r$ is bounded for $r \in [2,\frac{4(1+\alpha)}{2 \alpha+1})$.
First, for $r=2$ we have trivially
\begin{align*}
\|u\|_{X_2}^2 = \sum_{n \in \N} \sum_{l=1}^{d_n} \| u_{n,l} \|_{L^2_d(\R)}^{2}  \leq \sum_{n \in \N} \sum_{l=1}^{d_n} \sqrt{\nu_n} \|u_{n,l}\|_{L^2_d(\R)}^2 + \frac{1}{\sqrt{\nu_n}} \|u_{n,l}' \|_{L^2(\R)}^2 = \|u\|_{H}^2 .
\end{align*}
Now let $r \in (2,\frac{4(1+\alpha)}{2 \alpha+1}) $ and $v \in H^1(\R)$. Then by Gagliardo-Nirenberg inequality we have
\begin{align*}
\| v \|_{L^r_d(\R)}\leq c_{GN} \|v'\|_{L^2(\R)}^\theta \|v\|_{L^2_d(\R)}^{1- \theta}
\end{align*}
where $\theta=\frac{1}{2}-\frac{1}{r}$. Together with Hölder's inequality with the exponents $\frac{r+2}{4(r-1)}+ \frac{r-2}{4(r-1)}+\frac{r-2}{2(r-1)}=1$ we conclude
\begin{align*}
\|u\|_{X_r}^{r'}&= \sum_{n \in \N}\sum_{l=1}^{d_n} \norm{\frac{u_{n,l}}{s_{n,l}} }_{L^r_d(\R)}^{r'} s_{n,l}^2 \leq c_{GN} \sum_{n \in \N} \sum_{l=1}^{d_n} \norm{\frac{u_{n,l}}{s_{n,l}} }_{L^2_d(\R)}^{(1- \theta)r'} \norm{ \frac{u_{n,l}'}{s_{n,l}} }_{L^2(\R)}^{\theta r'} s_{n,l}^2 \\
&= c_{GN} \sum_{n \in \N} \sum_{l=1}^{d_n} \norm{\frac{u_{n,l}}{s_{n,l}} }_{L^2_d(\R)}^{\frac{r+2}{2(r-1)}} \norm{\frac{u_{n,l}'}{s_{n,l}} }_{L^2(\R)}^{\frac{r-2}{2(r-1)}} s_{n,l}^2 \\
&= c_{GN} \sum_{n \in \N} \sum_{l=1}^{d_n} \norm{\frac{u_{n,l}}{s_{n,l}} }_{L^2_d(\R)}^{\frac{r+2}{2(r-1)}} (\nu_n)^{\frac{r+2}{8(r-1)}} \norm{\frac{u_{n,l}'}{s_{n,l}} }_{L^2(\R)}^{\frac{r-2}{2(r-1)}} (\nu_n)^{-\frac{r-2}{8(r-1)}} s_{n,l}^2 (\nu_n)^{-\frac{1}{2(r-1)}} \\
&\leq c \left( \sum_{n \in \N} \sum_{l=1}^{d_n}  \norm{\frac{u_{n,l}}{s_{n,l}}}_{L^2_d(\R)}^{2} \sqrt{\nu_n} s_{n,l}^2 \right)^{\frac{r+2}{4(r-1)}} \left( \sum_{n \in \N} \sum_{l=1}^{d_n}  \norm{\frac{u_{n,l}'}{s_{n,l}}}_{L^2(\R)}^{2} (\sqrt{\nu_n})^{-1} s_{n,l}^2 \right)^{\frac{r-2}{4(r-1)}} \\
& \ \ \ \ \cdot \left( \sum_{n \in \N} d_n (\nu_n)^{-\frac{1}{r-2}} s_{n,l}^2 \right)^{\frac{r-2}{2(r-1)}}\\
&\leq \|u\|_H^{\frac{r}{r-1}} \left( \sum_{n \in \N}  d_n (\nu_n)^{-\frac{1}{r-2}} s_{n,l}^2 \right)^{\frac{r-2}{2(r-1)}} .
\end{align*}
By our assumptions \ref{as:efestimateM} and \ref{as:evM} we have 
\begin{align*}
    \sum_{n \in \N} d_n  (\nu_n)^{-\frac{1}{r-2}} s_{n,l}^2  \leq c\sum_{n \in \N}  (\nu_n)^{-\frac{1}{r-2}} \nu_n^\alpha 
\end{align*}
which is finite if $-\frac{1}{r-2}+\alpha < - \frac{1}{2}$, i.e. $r< \frac{4(1+\alpha)}{2 \alpha+1}$.

Step 3: We show that $X_r \hookrightarrow L^r_{Vd}(\R^N \times \R)$ is bounded for $r\in [2,\infty]$.
First let $r=r'=2$. Then by Parseval's equality we obtain
\begin{align*}
\|u\|_{X_2}^2=  \sum_{n \in \N} \sum_{l=1}^{d_n} \| u_{n,l} \|_{L^2_d(\R)}^{2}  =\|u\|_{L^2_{Vd}}^2 .
\end{align*}
Next, for $r= \infty$ and $r'=1$ we estimate
\begin{align*} 
\|u\|_{L^\infty} &=\esssup_{(x,t) \in \R^N\times\R}|u(x,t)| =\esssup_{(x,t) \in \R^N\times\R} \left| \sum_{n \in \N}  \sum_{l=1}^{d_n}\frac{u_{n,l}(t)}{s_{n,l}} \varphi_{n,l}(x)s_{n,l}\right| \\
&\leq \esssup_{(x,t) \in \R^N\times\R} \sum_{n \in \N}  \sum_{l=1}^{d_n}\left|\frac{u_{n,l}(t)}{s_{n,l}}\right| |\varphi_{n,l}(x)||s_{n,l}| \\
&\leq \esssup_{t \in \R}\sum_{n \in \N} \sum_{l=1}^{d_n} \left|\frac{u_{n,l}(t)}{s_{n,l}}\right| |s_{n,l}|^2 \\
&\leq \sum_{n \in \N} \sum_{l=1}^{d_n} \esssup_{t \in \R}\left|\frac{u_{n,l}(t)}{s_{n,l}}\right| |s_{n,l}|^2 \\
&= \|u\|_{X_\infty} .
\end{align*}
By Riesz-Thorin interpolation theorem $X_r \to L^r_{Vd}$ is bounded for all $r \in [2, \infty]$.
\end{proof}
\begin{proof}[Proof of Theorem \ref{lem:cct}]
We show the result for $q=2$ and $\tilde{q}\in(2,\frac{2 \alpha+3}{1+\alpha})$.
This then implies, together with Hölder's inequality and the boundedness of the embedding from Theorem~\ref{theom:emb}, the result for the remaining cases of $q$ and $\tilde{q}$.

In general, for $2<s<r< \frac{4(1+\alpha)}{2 \alpha+1}$ and $\theta \in (0,1)$ satisfying $\frac{1}{s}=\frac{\theta}{r}+\frac{1-\theta}{2}$, we have by Hölder's inequality
\begin{align*}
    \|u\|_{L^s}^s \leq \|u\|_{L^r}^{\theta s} \|u\|_{L^2}^{(1-\theta) s} .
\end{align*}
Now, for $r \in (2, \frac{4(1+\alpha)}{2 \alpha +1}) $ and $s=\frac{4(r-1)}{r}$, we have $2<s<r<\frac{4(1+\alpha)}{2 \alpha +1}$, $2<s<\frac{2\alpha+3}{1+\alpha}$ and $\theta=\frac{2}{s}$.

Therefore
\begin{align*}
    \|u\|_{L^s_V(\R^N \times \R)}^s &= \sum_{m \in \Z} \|u\|_{L^s_V(\R^N \times [mT,(m+1)T)]}^s \\
    &\leq \sup_{m \in \Z} \|u\|_{L^2_V(\R^N \times [mT,(m+1)T])}^{s-2} \sum_{m \in \Z} \|u\|_{L^r_V(\R^N \times [mT,(m+1)T])}^2.
\end{align*}
It remains to prove 
\begin{align*}
    \sum_{m \in \Z} \|u\|_{L^r_V(\R^N \times [mT,(m+1)T])}^2 \leq \| u \|_\Hs^2
\end{align*}
for $r \in (2, \frac{4(1+\alpha)}{2 \alpha +1})$ and $u \in \mathcal{H}$. 
For that, take $\phi_m \in C_c^\infty(\R)$, where $m \in \Z$, such that
\begin{align*}
\phi_m(t)= \begin{cases} 1, & t \in [mT,(m+1)T], \\
0, & t \not \in [(m-1)T, (m+2)T], \\
\in [0,1], & \text{elsewhere}
\end{cases}
\end{align*}
and $0 \leq |\phi_m'| \leq \frac{2}{T}$. Then, by Lemma \ref{lem:tripemb}
\begin{align*}
 &\sum_{m \in \Z} \|u\|_{L^r_V(\R^N \times [mT,(m+1)T])}^2 \leq \sum_{m \in \Z} \|u \phi_m\|_{L^r_{Vd}}^2  
\leq \sum_{m \in \Z} \|u \phi_m\|_{H}^2  \\
&=\sum_{m \in \Z} \left( \sum_{n \in \N} \sum_{l=1}^{d_n} \sqrt{\nu_n} \|u_{n,l} \phi_m\|_{L^2_d(\R)}^2 + \frac{1}{\sqrt{\nu_n}} \|(u_{n,l} \phi_m)' \|_{L^2(\R)}^2 \right) \\
&=\sum_{m \in \Z} \left( \sum_{n \in \N} \sum_{l=1}^{d_n} \sqrt{\nu_n} \|u_{n,l} \phi_m\|_{L^2_d(\R)}^2 + \frac{1}{\sqrt{\nu_n}} \|u_{n,l}' \phi_m +u_{n,l} \phi_m' \|_{L^2(\R)}^2 \right) \\
&\leq 3 \sum_{n \in \N} \sum_{l=1}^{d_n}\sqrt{\nu_n} \|u_{n,l} \|_{L^2_d(\R)}^2 + \sum_{m \in \Z} \sum_{n \in \N} \sum_{l=1}^{d_n} \frac{2}{\sqrt{\nu_n}} (  \|u_{n,l}' \phi_m  \|_{L^2(\R)}^2 + \|u_{n,l} \phi_m' \|_{L^2(\R)}^2) \\
&\leq c \|u\|_H^2 \\
&\leq \|u\|_\Hs^2 .
\end{align*}
\end{proof}
  %  \input{zeugs}
 %%%%%%%%%%%%%%%%%%%%%%%%%%%%%%%%%%%%%%%%%%%%%%%%%%%%%%%%%%%%%%  
 
	%\input{functional_calc.tex}

	%\appendix

	%\input{vector_L2.tex}

	%\input{eigenfunction_bounds.tex}

    %\input{examples}

	%\todoin{Nicht vergessen, hilfreiche Personen zu acknowledgen! (Plum, Robert, Xian)}

%	Funded by the Deutsche Forschungsgemeinschaft (DFG, German Research Foundation) – Project-ID 258734477 – SFB 1173
  \section*{Acknowledgments}
  The author is grateful to W. Reichel and S. Ohrem for their valuable feedback and support of this work.
	Funded by the Deutsche Forschungsgemeinschaft (DFG, German Research Foundation) – Project-ID 258734477 – SFB 1173.
	\printbibliography
\end{document}